\documentclass[reqno,11pt]{amsart}

\usepackage{mathdots}
\usepackage{tikz}
\usepackage{tikz-cd}
\usepackage{amssymb}
\usepackage{amsgen}
\usepackage{amsmath}
\usepackage{amsthm}
\usepackage{mathrsfs}
\usepackage{cite}
\usepackage{amsfonts}
\usepackage{enumitem}

\newcommand{\End}{\mathop{\mathrm{End}}}

\newcommand{\pd}{\mathop{\mathrm{pd}}}

\newcommand{\rad}{\mathop{\mathrm{rad}}\nolimits}

\newcommand{\Ind}{\mathop{\mathrm{Ind}}\nolimits}
\newcommand{\Coind}{\mathop{\mathrm{Coind}}\nolimits}

\newcommand{\J}{\mathrel{\mathscr J}} % J - relation
\newcommand{\R}{\mathrel{\mathscr R}} % R - relation
\newcommand{\eL}{\mathrel{\mathscr L}} % L - relation

\newcommand{\inv}{^{-1}}
\newcommand{\p}{\varphi}
\newcommand{\pinv}{{\p \inv}}
\newcommand{\ov}[1]{\ensuremath{\overline {#1}}}
\newcommand{\til}[1]{\ensuremath{\widetilde {#1}}}

\newcommand{\Hom}{\mathop{\mathrm{Hom}}\nolimits}
\newcommand{\Tor}{\mathop{\mathrm{Tor}}\nolimits}
\newcommand{\Ext}{\mathop{\mathrm{Ext}}\nolimits}

\usepackage{xcolor}

\newtheorem{Thm}{Theorem}[section]
\newtheorem{Prop}[Thm]{Proposition}

\newtheorem{Lemma}[Thm]{Lemma}
{\theoremstyle{definition}
}
{\theoremstyle{remark}
\newtheorem{Rmk}[Thm]{Remark}}
\newtheorem{Cor}[Thm]{Corollary}
{\theoremstyle{remark}
}
{\theoremstyle{remark}
}

{\theoremstyle{remark}
}
{\theoremstyle{remark}
}
{\theoremstyle{remark}
}
{\theoremstyle{remark}
\newtheorem*{Claim*}{Claim}}

\numberwithin{equation}{section}

\title{Topology and  monoid representations~I:  Foundations}

\author{Benjamin Steinberg}
\address[B.~Steinberg]{%
    Department of Mathematics\\
    City College of New York\\
    Convent Avenue at 138th Street\\
    New York, New York 10031\\
    USA}
\email{bsteinberg@ccny.cuny.edu}

\thanks{The author was supported by a PSC CUNY grant, a Simons Foundation Collaboration Grant, award number 849561, and the Australian Research Council Grant DP230103184.}
\date{\today}

\begin{document}

\begin{abstract}
This paper aims to use topological methods to compute $\mathrm{Ext}$ between an irreducible representation of a finite monoid inflated from its group completion and one inflated from its group of units, or more generally coinduced from a maximal subgroup, via a spectral sequence that collapses on the $E_2$-page over fields of good characteristic. 
 As an application, we determine the global dimension of the algebra of the monoid of all affine transformations of a vector space over a finite field.  We provide a topological characterization of when a monoid homomorphism induces a homological epimorphism of monoid algebras and apply it to semidirect products. Topology is used to construct projective resolutions of modules inflated from the group completion for sufficiently nice monoids. A sequel paper will use these results to study the representation theory Hsiao's monoid of ordered $G$-partitions (connected to the Mantaci-Reutenauer descent algebra for the wreath product $G\wr S_n$).
\end{abstract}
\maketitle
\section{Introduction}
The cohomology of a group $G$ with coefficients in a $KG$-module $V$ can be computed as the cohomology of an Eilenberg-Mac Lane space for $G$ with respect to a system of local coefficients. If $K$ is a field, then $\mathrm{Ext}$ between any two $KG$-modules can be computed as group cohomology and hence topologically.  For monoids the situation is quite different.  The cohomology of a monoid $M$ with coefficients in a $KM$-module $V$ can be computed as the cohomology of its classifying space with respect to a system of local coefficients only when $M$ acts by invertible maps on $V$, that is, when $V$ is a $KG(M)$-module where $G(M)$ is the group completion of $M$.  Moreover, even when $K$ is a field, there is in general no hope to compute $\Ext$ between arbitrary modules as monoid cohomology.  The global dimension of $KM$ can be much larger than the projective dimension of the trivial $KM$-module. We remark here that by the cohomology of a monoid, we mean the classical Eilenberg-Mac Lane cohomology.  There is a more sophisticated cohomology theory for monoids, due to Leech~\cite{Leechcohom}, that provides better structural information, but our concern is primarily with the monoid algebra $KM$ and Leech coefficients systems are more general than $KM$-modules, whereas Eilenberg-Mac Lane cohomology has coefficients which are $KM$-modules.

Nonetheless, the author together with Margolis and Saliola showed in~\cite{MSS} that topological methods can be used to compute $\Ext$ between certain simple modules for finite monoids.  If $M$ is a finite monoid, then in addition to the standard trivial module $K$, there is a second ``trivial'' module with underlying vector space $K$ on which invertible elements of $M$ act trivially and noninvertible elements of $M$ act as $0$.  In the case that $M$ has a trivial group of units, we showed that $\Ext$ between these two modules can be computed as the reduced cohomology of a certain infinite dimensional CW complex, up to a dimension shift.  Moreover, if $M$ is von Neumann regular, we showed that this CW complex can be replaced by the order complex of the (finite) poset of proper principal right ideals of $M$.  We then used this result to give a topological description of $\Ext$ between arbitrary simple modules for the algebra of a  left regular band over a field (this was developed further in~\cite{ourmemoirs}).  The representation theory of left regular bands had been used earlier to analyze a number of important Markov chains~\cite{BHR,DiaconisBrown1,Brown1,bjorner2} and is connected with Solomon's descent algebra~\cite{SolomonDescent,BidigareThesis,Brown2,SaliolaDescent}.

Recently, Khovanov \emph{et al.}~\cite{Khovanov} studied $\Ext^1$ between these two trivial modules for an arbitrary finite monoid $M$.  Although they use a different language and state a weaker result~\cite[Lemma~2B.5]{Khovanov}, an analysis of their proof (or some of the proofs in~\cite{DO}) shows that $\Ext^1$ can be interpreted as the $G$-invariants of the degree $0$ reduced cohomology of the order complex of the poset of proper principal right ideals of $M$, where $G$ denotes the group of units of $M$, which acts simplicially on this order complex;  see~\cite[Remark~2B.12]{Khovanov}.    In this paper, we will show that arbitrary $\Ext$ between these two modules can be interpreted as the relative cohomology of a CW pair.

The main results of this paper improve dramatically on the results of our previous work with Margolis and Saliola~\cite{MSS,ourmemoirs}, which were restricted to an important, but quite special class of monoids.  Here,
we are able to compute $\Ext$ between $KM$-modules $V,W$ using topological methods under the hypothesis that $V$ is inflated from the group completion of $M$ and $W$ is inflated from the group of units $G$ of $M$ (or more generally when $W$ is induced or coinduced from a maximal subgroup $G$). Our main result is the existence of a spectral sequence relating $\Ext_{KM}^n(V,W)$ to the homology of a certain $G$-CW complex.  If the characteristic  of $K$ does not divide $|G|$, the spectral sequence collapses on the $E_2$-page on the line $p=0$, and then the computations boil down to decomposing homology representations of $G$ coming from a $G$-CW complex.  In the case that $M$ is regular, the homology representations come from the simplicial action of $G$ on the order complex of a subposet of the poset of principal right ideals.

We also introduce topological techniques to determine whether a monoid homomorphism $\p\colon M\to N$ induces a homological  epimorphism $KM\to KN$ of monoid algebras.  A homological epimorphism is roughly speaking a ring epimorphism  $R\to S$ for which $\Ext$ and $\Tor$ of $S$-modules are the same whether taken over $R$ or $S$.  In particular, we characterize when a semidirect product projection $M\rtimes N\to N$ induces a homological epimorphism in terms of the topology of the classifying space of $M$.

We give several applications of these results in this paper.  Further applications will appear in a sequel paper, where we generalize all of the results of~\cite{MSS,ourmemoirs} to left regular bands of groups via topological means.

We use our topological methods to compute the global dimension of the algebra of the monoid of all affine transformations of an $n$-dimensional vector space over a finite field.  A key role is played by the nontrivial results of Okni\'nski and Putcha~\cite{putchasemisimple} and of Kov\'acs~\cite{Kovacs} on the representation theory of the monoid of $n\times n$ matrices over a finite field, and also by Solomon's work on the analogue of the Steinberg representation for the affine general linear group~\cite{SolomonAffine,SolomonAffineBruhat}.  We also compute the monoid cohomology of any simple module for the affine monoid over the field of complex numbers.

A number of our results work over more general coefficient rings than fields and for more general monoids than finite ones.  Therefore, we impose these assumptions only when necessary. The reader interested in only the case of fields and finite monoids can then skip arguments that trivialize under these assumptions.

The paper is structured as follows.  We begin with a section recalling some homological results that we shall apply a number of times.  Then in Section~\ref{s:sec3} we recall basic notions concerning simplicial sets and classifying spaces of categories before applying the techniques to monoid actions, expanding on results of Nunes~\cite{Nunes} and Margolis, Saliola and the author~\cite{MSS}. Section~\ref{s:sec4} proves our main result concerning $\Ext$ between modules inflated from the group completion and modules induced or coinduced from a maximal subgroup and provides the topological characterization of homological epimorphisms.  We apply the results to determine when the projection from a semidirect product (or more generally a crossed product) induces a homological epimorphism.
The next section, computes the global dimension of the algebra of the monoid of all affine transformations of a vector space over a finite field and computes arbitrary $\Ext$ between simple modules from  two families.  We end with a section computing resolutions of modules inflated from the group completion of a regular monoid by modules which are filtered by standard modules with respect to the canonical quasi-hereditary algebra structure on the algebra of a regular monoid (in good characteristic).  The chain complex is, in fact, the simplicial chain complex of the poset of principal right ideals.  For certain monoids, like the monoid of affine transformations, all standard modules are projective, and so this yields a projective resolution.

\section{The Adams-Rieffel theorem and monoid cohomology}
The following result was proved by Adams and Rieffel~\cite{AdamsRieffel}, in part motivated by monoid cohomology.  It generalizes the classical Eckmann-Shapiro lemma.  We sketch the proof for completeness, as the result is fundamental for us.  Also we state it in the context of $K$-linear abelian categories.

\begin{Thm}\label{t:adams.rieffel}
Let $K$ be a commutative ring and $\mathscr A,\mathscr B$ $K$-linear abelian categories such that $\mathscr B$ has enough injectives.  Suppose that $F\colon \mathscr A\to \mathscr B$ and $G\colon \mathscr B\to \mathscr A$ are $K$-linear adjoint functors such that $\Hom(F(-),-)\cong \Hom (-,G(-))$ and $F,G$ are exact.  Then \[\mathrm{Ext}^n(F(-),-)\cong \mathrm{Ext}^n(-,G(-))\] as $K$-modules for all $n\geq 0$.
\end{Thm}
\begin{proof}
First note that if $I$ is an injective object in $\mathscr B$, then $\Hom(-,G(I))\cong \Hom(F(-),I)$ is exact, as $F$ is exact and $I$ is injective. Therefore, $G(I)$ is injective, that is, $G$ preserves injectives.  Thus if $B\to I_\bullet$ is a injective resolution of $B\in \mathscr B$, then $G(B)\to G(I_{\bullet})$ is an injective resolution as $G$ is exact and preserves injectives.  Thus if $A\in \mathscr A$, then $\Ext^n(A,G(B))\cong H^n(\Hom(A,G(I_{\bullet}))) \cong H^n(\Hom(F(A),I_{\bullet}))\cong \Ext^n(F(A),B)$.  This completes the proof.
\end{proof}

We consider two special cases, one of which appears in~\cite{AdamsRieffel}. These will play an important role in our sequel paper. 
If $e\in R$ is an idempotent with $eR=eRe$, then $eRe$ is a unital ring with identity $e$ and $\rho\colon R\to eR$ given by $\rho(r) = er$ is a surjective homomorphism fixing $eR$.   Hence every $eRe$-module $B$ can be viewed as an $R$-module via $rb=erb$ for $r\in R$ and $b\in B$ (and this agrees with the original action on $eRe$).   Also, if $A$ is an $R$-module, then $eA$ is an $eRe$-module by restricting the action.  Dual remarks apply if $Re=eRe$.

\begin{Cor}\label{c:rings.eRe}
Let $R$ be a $K$-algebra and $e\in R$. Let $A$ be an $R$-module and $B$ an $eRe$-module.
\begin{enumerate}
\item If $Re=eRe$,  then
$\Ext^n_{R}(B,A)\cong \Ext^n_{eRe}(B,eA)$ as $K$-modules.
\item If $eR=eRe$,  then $\Ext^n_R(A,B)\cong \Ext^n_{eRe}(eA,B)$ as $K$-modules.
\end{enumerate}
\end{Cor}
\begin{proof}
Assume first that $Re=eRe$.   The functor $A\mapsto eA$ from $R$-modules to $eRe$-modules is isomorphic to $\Hom_R(Re, -)$ and is hence exact because $Re$ is projective. It thus has left adjoint $Re\otimes_{eRe} -$.  Since $Re=eRe$ is a free $eRe$-module on one generator, this left adjoint is exact and is isomorphic to the inclusion of the category of $eRe$-modules into $R$-modules coming from the retraction $R\to Re$.  The Adams-Rieffel theorem then yields $\Ext^n_R(B,A)\cong \Ext^n_{eRe}(B,eA)$ for all $n\geq 0$.

Dually, if $eR=eRe$, then the exact functor $A\mapsto eA$ is isomorphic to $eR\otimes_R -$ and so has right adjoint $\Hom_{eRe}(eR,-)$.  But since $eR=eRe$, it is a free $eRe$-module on one generator, and so $\Hom_{eRe}(eR,-)$ is exact and is isomorphic to the inclusion of the category of $eRe$-modules into that of $R$-modules coming from the retraction $R\to eR$.  Thus, by the Adams-Rieffel theorem, we have that $\Ext^n_{eRe}(eA,B)\cong \Ext^n_R(A,B)$.
\end{proof}

If $K$ is a commutative ring and $M$ is a monoid, then $KM$ denotes the monoid algebra of $M$ over $K$.  We shall use $K$ to denote the trivial $KM$-module.  The cohomology of $M$ with respect to a $KM$-module $V$ is defined by $H^n(M,V) = \Ext^n_{KM}(K,V)$. Monoid homology for a right $KM$-module $V$ is defined dually by $H_n(M,V) = \Tor_n^{KM}(V,K)$.  We remark that if $K$ is an $L$-algebra, then, for any $KM$-module $V$, one has that $\Ext^n_{KM}(K,V)\cong \Ext^n_{LM}(L,V)$ as $L$-modules, and so $H^n(M,V)$ is unambiguous, and similarly for $H_n(M,V)$ for a right $KM$-module.  This follows from the standard bar resolutions used to compute monoid (co)homology (cf. Corollary~\ref{c:standard.res} below).

The first item of the following corollary is in~\cite{AdamsRieffel} and the second is essentially in~\cite{rrbg}.  See also~\cite{Nicocohom1}.

\begin{Cor}\label{c:monoid.er}
Let $M$ be a monoid, $K$ a commutative ring and $e\in M$ an idempotent.
\begin{enumerate}
\item If $Me=eMe$, then $H^n(M,A)\cong H^n(eMe,eA)$ for any $KM$-module $A$.
\item If $eM=eMe$, then $H^n(M,A)\cong H^n(eMe,A)$ for any $K[eMe]$-module $A$.
\end{enumerate}
\end{Cor}
\begin{proof}
The first item follows from Corollary~\ref{c:rings.eRe}(1) after noting that the trivial $K[eMe]$-module inflates to the trivial $KM$-module.  The second item follows from Corollary~\ref{c:rings.eRe}(2) after noting that if $K$ is the trivial $KM$-module, then $eK$ is the trivial $K[eMe]$-module.
\end{proof}

The last homological result that we shall need is a well-known spectral sequence~\cite[Theorem~10.62]{Rotmanhom} and its corollary~\cite[Corollary~10.63]{Rotmanhom}.

\begin{Thm}\label{t:ext.dual.tor}
Let $R,S$ be $K$-algebras and suppose that $B$ is an $S$-$R$-bimodule with the same $K$-action on both sides (i.e., an $S\otimes_K R^{op}$-module).  Assume that $B,C$ satisfy $\Ext_S^i(B\otimes_R P,C)=0$ for all $i\geq 1$ whenever $P$ is a projective $R$-module.  Then there is a spectral sequence of $K$-modules
\[\Ext^p_S(\Tor_q^R(B,A),C)\Rightarrow_p \Ext^n_R(A,\Hom_S(B,C)).\]
In particular,  if $C$ is an injective $S$-module, then there is an isomorphism of $K$-modules
\[\Ext^n_R(A,\Hom_S(B,C))\cong \Hom_S(\Tor_n^{R}(B,A),C)\] natural in the $R$-module $A$.
\end{Thm}

\section{Simplicial sets, classifying spaces and monoid actions}\label{s:sec3}

We consider here some basic tools from topology and category theory and apply them to monoid actions.  First of all, we adopt the convention that $\til H_{-1}(\emptyset,K)=K=\til H^{-1}(\emptyset,K)$ and $\til H_n(\emptyset, K)=0=\til H^n(\emptyset,K)$ for  $n\geq 0$. Of course, $\til H_{-1}(X,K)=0=\til H^{-1}(X,K)$ if $X\neq \emptyset$.  This corresponds to allowing singular $(-1)$-simplices  $\emptyset\to X$ for a space $X$.  The advantage of this convention is that the reduced version of the long exact sequence in relative (co)homology works for any pair of spaces $(X,A)$, including $A=\emptyset$.

\subsection{Simplicial sets}
We briefly recall some basic facts about simplicial sets and their geometric realizations.   Details can be found in~\cite{may,jardine}.

If $n\geq 0$, put $[n]=\{0,\ldots, n\}$.  The \emph{simplex category} $\Delta$ is the category with objects $[n]$ with $n\geq 0$ and arrows the order-preserving maps. A \emph{simplicial set} is a contravariant functor $X$ from $\Delta$ to the category of sets.   The simplicial set $X$ is determined by the data $X_n=X([n])$, the face maps  $d_i\colon X_n\to X_{n-1}$ for $i=0,\ldots,n$ and the degeneracy maps $s_i\colon X_n\to X_{n+1}$ for $i=0,\ldots, n$ corresponding to the images under $X$ of the injective mapping $[n-1]\to [n]$ missing $i$ from its image and the surjective mapping $[n+1]\to [n]$ collapsing $i,i+1$, respectively. These maps, of course, must satisfy certain identities.   Elements of $X_n$ are called \emph{$n$-simplices}.

A morphism of simplicial sets, called a \emph{simplicial mapping}, is just a natural transformation $f\colon X\to X'$ of functors.  Equivalently, it is collection of mappings $f_n\colon X_n\to X_n'$ that commute with the face and degeneracy maps.

An order-preserving mapping $[k]\to [n]$ induces a simplicial mapping of simplicial complexes $\Delta_k\to \Delta_n$ where $\Delta_q$ is the standard $q$-simplex.  There results a covariant functor $F$ from $\Delta$ to the category of compactly generated Hausdorff spaces.  If we view each set as a discrete compactly generated Hausdorff space, then we can take the tensor product (or Kan extension) of $F$ with a simplicial set $X$ to obtain a compactly generated Hausdorff space $|X|$ known as the \emph{geometric realization} of $X$.

Concretely, \[|X|=\left(\coprod_{n\geq 0} X_n\times \Delta_n\right)/{\sim}\] where $\sim$ is the least equivalence relation with $(X(f)(\sigma),t) \sim (\sigma,F(f)(t))$ for $f\colon [q]\to [n]$ order preserving, $\sigma\in X_n$ and $t\in \Delta_q$.  It is well known that $|X|$ is a CW complex whose closed $q$-cells are given by the images in $|X|$ of the $\{\sigma\}\times \Delta_q$ with $\sigma$ a non-degenerate $q$-simplex, that is, a simplex not in the image of any degeneracy.
Geometric realization is functorial with respect to simplicial mappings.

If $X$ is a simplicial set and $K$ is a commutative ring, one can build a chain complex $M_\bullet(X,K)$ of $K$-modules, called the \emph{Moore complex}, with  $M_n(X,K) = KX_n$ and with boundary map $\partial_n\colon KX_n\to KX_{n-1}$ given by $\partial(\sigma) = \sum_{i=0}^n (-1)^id_i(\sigma)$ for an $n$-simplex $\sigma$.  Let $D_n(X,K)$ be the submodule spanned by the degenerate $n$-simplices.  Then $D_\bullet(X,K)$ is a subcomplex of $M_\bullet(X,K)$ and $C_\bullet(X,K)=M_\bullet(X,K)/D_\bullet(X,K)$ is called the \emph{normalized chain complex} of $X$ with coefficients in $K$.  One then has that $C_\bullet(X,K)$ is isomorphic to the cellular chain complex of $|X|$ with coefficients in $K$ with respect to the cell structure described above. In particular, $H_n(C_\bullet(X,K))\cong H_n(|X|,K)$.  Also it is known that the projection $M_\bullet(X,K)\to C_\bullet(X,K)$ induces an isomorphism on homology.

Similarly, one can define a normalized cochain complex $C^\bullet(X,K)$ with $C^n(X,K)$ consisting of those mappings $f\colon X_n\to K$ vanishing on degenerate simplices and with coboundary the dual of the boundary map above.  One can identify $C^\bullet(X,K)$ with the cellular cochain complex of $|X|$ and hence $H^n(|X|,K)\cong H^n(C^\bullet(X,K))$.

\subsection{Classifying spaces of categories}
There is a simplicial set $\mathcal N\mathscr C$ associated to any category $\mathscr C$ called the \emph{nerve} of $\mathscr C$.  The geometric realization of $\mathcal N\mathscr C$ is denoted $\mathcal B\mathscr C$ and is called the \emph{classifying space} of $\mathscr C$.  See~\cite{GSegal,Quillen,Rosenberg} for details.  We shall sometimes call a $q$-simplex of $\mathcal N\mathscr C$ a $q$-simplex of $\mathcal B\mathscr C$.  Since a monoid $M$ can be viewed as a category with one object, we can also consider the classifying space $\mathcal BM$ of $M$.

Associated to any poset $P$ is a category with object set $P$ and a unique arrow $(p,q)$ whenever $p\leq q$ with the composition $(p,q)(r,p)=(r,q)$ for $r\leq p\leq q$; we shall, abusing notation, also denote this category by $P$.  Functors between posetal categories correspond precisely to order-preserving maps of posets.

 Since $[n]$ is a poset for each $n\geq 0$, given a fixed category $\mathscr C$, we have a simplicial set $\mathcal N\mathscr C$, given by $[n]\mapsto \Hom([n],\mathscr C)$, i.e.,  obtained by composing the inclusion of $\Delta$ into the category of categories with the hom functor into $\mathscr C$.  Concretely, a $0$-simplex of $\mathcal N\mathscr C$ is an object of $\mathscr C$ and a $q$-simplex of $\mathcal N\mathscr C$, for $q\geq 1$, is a sequence \[c_0\xrightarrow{\,f_1\,} c_1\xrightarrow{\,f_2\,}c_2\longrightarrow \cdots \longrightarrow c_{q-1}\xrightarrow{\,f_q\,}c_q\] of composable arrows, which we denote $(f_q,\ldots, f_1)$.  The face map $d_i$ omits the object $c_i$ and if $1\leq i\leq q-1$ composes $f_i$ and $f_{i+1}$.  The degeneracy map $s_i$  inserts the identity map $1_{c_i}$ at $c_i$.  The degenerate simplices are precisely those with some entry an identity.

For example, if $P$ is a poset, then a $q$-simplex corresponds to a sequence $p_0\leq p_1\leq\cdots\leq p_q$ of elements of $P$ and the $q$-simplex is non-degenerate if and only if the chain is strict.   From this, it is easy to see that any face of a non-degenerate simplex of $\mathcal NP$ is non-degenerate, a simplex is determined by its vertices and  $\mathcal BP$ is isomorphic as a CW complex to the order complex $\Delta(P)$  of $P$, that is,  the simplicial complex with vertex set $P$ and simplices the finite chains in $P$.

The nerve construction is functorial, taking a functor between categories to a simplicial map.  Indeed, if $F\colon \mathscr C\to \mathscr D$ and $(f_q,\ldots,f_1)$ is a $q$-simplex of $\mathcal N\mathscr C$, then one has that $(F(f_q),\ldots, F(f_1))$ is a $q$-simplex of $\mathcal N\mathscr D$.  The geometric realization of this simplicial mapping is denoted $\mathcal BF$.
 It is well known that $\mathcal B\mathscr C\cong \mathcal B\mathscr C^{op}$ for any category $\mathscr C$.

A key result of Segal~\cite{GSegal} is the following (cf.~\cite[Lemma~5.3.17]{Rosenberg}).

\begin{Lemma}[Segal]\label{l:SegalLemma}
If $F,G\colon \mathscr C\to \mathscr D$ are functors between small categories and there is a natural transformation $F\Rightarrow G$, then $\mathcal BF$ and $\mathcal BG$ are homotopic.
\end{Lemma}

In particular, if $\mathscr C$ and $\mathscr D$ are naturally equivalent, then
$\mathcal B\mathscr C$ and $\mathcal B\mathscr D$ are homotopy equivalent.  In fact, if a functor $F\colon \mathscr C\to \mathscr D$ has an adjoint, then $\mathcal B\mathscr C$ and $\mathcal B\mathscr D$ are homotopy equivalent~\cite[Corollary 3.7]{ktheory}. In particular,  a category with an initial or terminal object has a contractible classifying space.

One of the most powerful tools for working with classifying spaces of categories is Quillen's Theorem~A~\cite{Quillen}.  If $F\colon \mathscr C\to \mathscr D$ is a functor and $d$ is an object of $\mathscr D$, then the category $F/d$ has objects all pairs $(c,f)$ with $c$ an object of $\mathscr C$ and  $f\colon F(c)\to d$ an arrow of $\mathscr D$.  If $(c,f)$ and $(c',f')$ are objects of $F/d$, then a morphism from $(c,f)$ to $(c',f')$ is an arrow $g\colon c\to c'$ of $\mathscr C$ such that
\[\begin{tikzcd}
  F(c)\ar[rr, "F(g)"]\ar[rd,swap, "f"] & &  F(c')\ar[ld, "f'"]\\ &d&
  \end{tikzcd}\]
commutes. The composition of arrows is that of $\mathscr C$.

\begin{Thm}[Quillen's Theorem~A]
Let $F\colon \mathscr C\to \mathscr D$ be a functor such that $\mathcal B(F/d)$ is contractible for all objects $d$ of $\mathscr D$.  Then $\mathcal BF\colon \mathcal B\mathscr C\to \mathcal B\mathscr D$ is a homotopy equivalence.
\end{Thm}

\subsection{Classifying spaces of monoid actions}
Let $M$ be a monoid.  Then by a (right) \emph{$M$-set}, we mean a set $X$  equipped with a right action of $M$.  By an $M$-set pair $(X,Y)$, we mean an $M$-set $X$ and an $M$-invariant subset $Y$. There is an evident category of $M$-set pairs and we write $\End_M(X,Y)$ for the endomorphism monoid of $(X,Y)$ in this category; so it consists of all $M$-equivariant mappings $f\colon X\to X$ with $f(Y)\subseteq Y$.  Of course the automorphism group is denoted $\mathrm{Aut}_M(X,Y)$.

 If $K$ is a commutative ring, then associated to an $M$-set $X$ is a right $KM$-module $KX$ and a left $KM$-module $K^X$.  The former is defined by extending $K$-linearly the action of $M$ on $X$ and the latter by $(mf)(x) =f(xm)$ for $m\in M$ and $x\in X$.  If $(X,Y)$ is an $M$-set pair, then the set of mappings $f\colon X\to K$ that vanish on $Y$ is a $KM$-submodule of $K^X$ that we denote $K^{X\setminus Y}$ since such a mapping is uniquely determined by its restriction to $X\setminus Y$ (with inverse extension by $0$ to $Y$).  We then have exact sequences
\begin{gather*}
0\longrightarrow KY\longrightarrow KX\longrightarrow KX/KY\longrightarrow 0\\
0\longrightarrow K^{X\setminus Y}\longrightarrow K^X\longrightarrow K^Y\longrightarrow 0
\end{gather*}
where the second exact sequence comes from restricting mappings from $X$ to $Y$.   Note that $KX/KY$ has basis in bijection with $X\setminus Y$ and under this identification, the action becomes \[x\cdot m = \begin{cases} xm, & \text{if}\ xm\notin Y\\ 0, & \text{if}\ xm\in Y.\end{cases}\]  Thus we will often write $K[X\setminus Y]$ for $KX/KY$.  Notice that $K[X\setminus Y]$ is naturally a left $K\End_M(X,Y)$-module and $K^{X\setminus Y}$ is naturally a right $K\End_M(X,Y)$-module.

If $X$ is an $M$-set, then a basis for $X$ is a subset $T\subseteq X$ such that each $x\in X$ has a unique expression of the form $x=tm$ with $m\in M$ and $t\in T$.  An $M$-set with basis $T$ is said to be a free $M$-set on $T$.  For example, if $G$ is a group then a $G$-set is free if and only if all the point stabilizers are trivial, in which case any set of orbit representatives is a basis.   Note that $T\times M$ is a free $M$-set with basis $T\times \{1\}$ in bijection with $T$ via the action $(t,m_0)m = (t,m_0m)$. Moreover, if $X$ is an $M$-set and $T\subseteq X$, then $T$ is a basis for $X$ if and only if the canonical mapping $T\times M\to X$ given by $(t,m)\mapsto tm$ is an isomorphism of $M$-sets. Of course, dual notions apply to left $M$-sets.  If $X$ is a free (left) $M$-set with basis $T$, then $KX$ is a free (left) $KM$-module with basis $T$, a fact that we shall use frequently without comment.

Nunes~\cite{Nunes} used the category of elements of an $M$-set $X$ to study the homology of $KX$.  The corresponding theory for the cohomology of $K^X$ was considered by Margolis, Saliola and the author in~\cite{MSS}. Here we develop the theory for $M$-set pairs, which was only partially developed in~\cite{MSS}.

If $X$ is an $M$-set, then the \emph{category of elements} $X\rtimes M$ of $M$ has object set $X$ and arrow set $X\times M$.  The arrow $(x,m)$ goes from $xm$ to $x$ and composition is given by $(x,m)(xm,n)= (x,mn)$. The identity at $x$ is $(x,1)$. We note that the opposite of this category is used in~\cite{Nunes}, but we follow the convention from category theory.

Note that a $q$-simplex in the nerve of $X\rtimes M$ is of the form
\begin{gather*}
((x,m_q),(xm_q,m_{q-1}),\ldots, (xm_q\cdots m_2,m_1)), \\
xm_q\cdots m_1\xrightarrow{\,m_1\,} xm_q\cdots m_2\xrightarrow{\,m_2\,}\cdots\xrightarrow{m_{q-2}}xm_qm_{q-1}\xrightarrow{m_{q-1}}xm_q\xrightarrow{\,m_q\,} x
\end{gather*}
 and hence is uniquely determined by the $(q+1)$-tuple $(x,m_q,\ldots, m_1)$.  That is, the set of $q$-simplices of $\mathcal \mathcal B(X\rtimes M)$ can be identified with the set $X\times M^q$.  The degenerate simplices are precisely those containing an entry $m_i=1$.   We also record here that in the normalized chain complex $C_q(\mathcal B(X\rtimes M),K)$ one has
\begin{equation}\label{eq:boundary.cat.els}
\begin{split}
d(x,m_q,\ldots, m_1) &= (x,m_q,\ldots, m_2)+\sum_{i=1}^{q-1} (-1)^i(x,m_q,\ldots, m_{i+1}m_i,\ldots, m_1)\\  &\qquad +(-1)^q(xm_q,m_{q-1},\ldots, m_1)
\end{split}
\end{equation}
where degenerate simplices are treated as zero (so the above formula is the computation in the Moore complex without this convention).

We shall say that an $M$-set $X$ is \emph{contractible} if $\mathcal B(X\rtimes M)$ is contractible.  Notice that if $X$ is the one-point set, then $X\rtimes M$ is the one-object category $M$, and so $\mathcal B(X\rtimes M) = \mathcal BM$.  If $U(M)$ denotes the underlying set of $M$, then one puts $\mathcal EM=\mathcal B(U(M)\rtimes M)$.   Note that $M$ acts on the left of $U(M)\rtimes M$ via left multiplication on objects and $m(x,m') = (mx,m')$ on arrows.  Thus $M$ acts by simplicial mappings on the nerve $\mathcal N(U(M)\rtimes M)$ and hence cellularly on $\mathcal EM$.   The set of nondegenerate $q$-simplices of $\mathcal EM$ can be identified with $M\times (M\setminus \{1\})^q$ with action $m(m_q,m_{q-1},\ldots, m_0)= (mm_q,m_{q-1},\ldots, m_0)$ and hence is a free left $M$-set with basis $\{1\}\times (M\setminus \{1\})^q$, which is in bijection with $(M\setminus \{1\})^q$.  Therefore, an element of $M\times (M\setminus \{1\})^q$ will be written as $m(m_q,\ldots, m_1)$ instead of $(m,m_q,\ldots, m_1)$.     In fact,  $\mathcal EM$ is a free contractible $M$-CW complex and is uniquely determined by this property up to $M$-homotopy equivalence~\cite{GS1}.   For convenience of the reader, we prove the contractibility of $\mathcal EM$ and record its consequences, although this is not by any means new.

\begin{Prop}\label{p:EM}
Let $M$ be a monoid.  Then $\mathcal EM$ is contractible.  Moreover, if $K$ is any commutative ring, then a free resolution of the trivial module $K$ is given by the augmented normalized chain complex of $\mathcal EM$, that is, $C_q(\mathcal EM,K)$ is a free $KM$-module with basis $(M\setminus \{1\})^q$. The boundary map is given on the basis by
\begin{align*}
d_q(m_q,\ldots, m_1) &= (m_q,\ldots, m_2) +\sum_{i=1}^{q-1} (-1)^i(m_q,\ldots, m_{i+1}m_i,\ldots, m_1)\\ &\qquad +(-1)^{q}m_q(m_{q-1},\ldots, m_1)
\end{align*}
where degenerate simplices are treated as $0$,  and the augmentation sends $m\in M$ to $1$ (identifying the free $KM$-module on $M^0$ with $KM$).
\end{Prop}
\begin{proof}
The category $U(M)\rtimes M$ has terminal object $1$ since $(1,m)$ is the unique arrow from $m$ to $1$. Thus
 $\mathcal EM$ is contractible, and so
 the augmented normalized chain complex of $\mathcal EM$ is
 a resolution of $K$.  The formula for the boundary map is immediate from~\eqref{eq:boundary.cat.els} with $x=1$.
\end{proof}

We may then deduce the standard chain and cochain complexes for monoid homology and cohomology and ``independence'' from the ground ring.

\begin{Cor}\label{c:standard.res}
Let $M$ be a monoid and $K$ a commutative ring with unit.
\begin{enumerate}
\item If $V$ is a left $KM$-module, then $H^q(M,V)$ is the $q^{th}$-cohomology of the cochain complex $C^\bullet(M,V)$ with $C^q(M,V) = V^{(M\setminus \{1\})^q}$ (identified with those mappings in $V^{M^q}$ vanishing on $q$-tuples containing $1$) with coboundary map $\delta\colon C^q(M,V)\to C^{q+1}(M,V)$ given by
    \begin{align*}
    \delta(f)(m_q,\ldots, m_0) &= f(m_q,\ldots, m_1)+\sum_{i=0}^{q-1}(-1)^{i+1} f(m_q,\ldots, m_{i+1}m_i,\cdots m_0)\\ &\qquad +(-1)^{q+1}m_qf(m_{q-1},\cdots, m_0).
    \end{align*}
      In particular, if $K$ is an $L$-algebra, then $\Ext^q_{KM}(K,V)\cong \Ext^q_{LM}(L,V)$ as $L$-modules.
\item If $V$ is a right $KM$-module, then $H_q(M,V)$ is the $q^{th}$-homology of the chain complex $C_{\bullet}(M,V)$ where $C_q(M,V) = V\otimes_K K[(M\setminus \{1\})^q]$.  The boundary map $\delta\colon C_q(M,V)\to C_{q-1}(M,V)$ is given by
    \begin{align*}
    d(v\otimes (m_q,\ldots, m_1)) &= v\otimes (m_q,\ldots, m_2)+ \sum_{i=1}^{q-1}(-1)^iv\otimes (m_q,\ldots, m_{i+1}m_i,\ldots, m_1)\\ &\qquad  +(-1)^qvm_q\otimes (m_{q-1},\ldots, m_1)
    \end{align*}
     where we interpret $v\otimes (a_{q-1},\ldots, a_1)=0$ if some $a_i=1$.   In particular, if $K$ is an $L$-algebra, then $\Tor_q^{KM}(V,K)\cong \Tor_q^{LM}(V,L)$ as $L$-modules.
\end{enumerate}
\end{Cor}

Another corollary of Proposition~\ref{p:EM} is the following (which appears in~\cite{MSS}).

\begin{Cor}\label{c:contract.proj}
Let $M$ be a monoid and $e\in M$ an idempotent.  Then $eM$ is a contractible right $M$-set, that is, $\mathcal B(eM\rtimes M)$ is contractible.
\end{Cor}
\begin{proof}
Note that left multiplication by $e$ induces a retraction of $M$-sets of $U(M)$ onto $eM$.  Thus $\mathcal B(eM\rtimes M)$ is a retract of $\mathcal EM$ and hence contractible as a retract of a contractible space is contractible.
\end{proof}

The first part of the following proposition can be found in~\cite{Nunes} (see also the appendix of~\cite{Loday} for a corresponding result for categories),  and the second part is essentially in~\cite{MSS} at the level of cohomology.  The final part seems to be new.

\begin{Prop}\label{p:pass.to.relative}
Let $M$ be a monoid and $K$ a commutative ring with unit.
\begin{enumerate}
\item There is an isomorphism $C_{\bullet}(M,KX)\cong C_{\bullet}(\mathcal B(X\rtimes M),K)$ for each $M$-set $X$, natural in $X$.
\item  There is an isomorphism $C^{\bullet}(M,K^X)\cong C^{\bullet}(\mathcal B(X\rtimes M),K)$ for each $M$-set $X$, natural in $X$.
\item If $(X,Y)$ is an $M$-set pair, then there are isomorphisms, natural in the pair $(X,Y)$, of (co)chain complexes
\begin{gather*}
C_{\bullet}(M,K[X\setminus Y])\to C_{\bullet}(\mathcal B(X\rtimes, M),\mathcal B(Y\rtimes M),K),\ \text{and}\\\ C^\bullet(M, K^{X\setminus Y})\to C^\bullet(\mathcal B(X\rtimes M), \mathcal B(Y\rtimes M),K),\end{gather*} where the relative (co)homology is with respect to the natural CW structure on a classifying space.  In particular, there are isomorphisms of $K\End_M(X,Y)$-modules
\begin{align*}
H_n(M, K[X\setminus Y])& \cong H_n(\mathcal B(X\rtimes M),\mathcal B(Y\rtimes M),K)\\  H^n(M, K^{X\setminus Y}) &\cong H^n(\mathcal B(X\rtimes M),\mathcal B(Y\rtimes M),K)
\end{align*}
for all $n\geq 0$.
\end{enumerate}
\end{Prop}
\begin{proof}
A $K$-basis for $C_q(M,KX)$ consists of all elements of the form $x\otimes (m_q,\ldots, m_1)$ with $x\in X$ and $m_1,\ldots, m_q\in M\setminus \{1\}$, and a $K$-basis for $C_q(\mathcal B(X\rtimes M),K)$ consists of all elements of the form $(x,m_q,\ldots, m_1)$ with $x\in X$ and $m_1,\ldots, m_q\in M\setminus \{1\}$.  It's immediate from \eqref{eq:boundary.cat.els} and Corollary~\ref{c:standard.res}(2) that the maps sending $x\otimes (m_q,\ldots, m_1)$ to $(x,m_q,\ldots, m_1)$ for $q\geq 0$ yield an isomorphism of chain complexes.  Naturality of the isomorphism is clear.

We have that $C^q(M,K^X) = (K^X)^{(M\setminus \{1\})^q}$ and $C^q(\mathcal B(X\rtimes M),K) = K^{X\times (M\setminus \{1\})^q}$. These are isomorphic $K$-modules via currying, that is, the map sending $f\in (K^X)^{(M\setminus \{1\})^q}$ to $F\in K^{X\times (M\setminus \{1\})^q}$ with $F(x,m_q,\ldots, m_1) = f(m_q,\ldots, m_1)(x)$. A routine computation shows that this assignment is a chain map, natural in $X$.

If $(X,Y)$ is an $M$-set pair, then it is clear that the currying isomorphism $C^q(M,K^X)\to C^q(\mathcal B(X\rtimes M),K)$ has the property that $f(m_1,\ldots, m_q)$ vanishes on $Y$ for all $m_1,\ldots, m_q\in M\setminus \{1\}$, if and only if the associated mapping $F\in K^{X\times (M\setminus \{1\})^q}$ vanishes on  $Y\times (M\setminus \{1\})^q$, and hence it restricts to an isomorphism of cochain complexes $C^\bullet(M,K^{X\setminus Y})\to C^\bullet(\mathcal B(X\rtimes, M), \mathcal B(Y\rtimes M),K)$.  The naturality is immediate.

Note that  $C_q(M,KX)/C_q(M,KY)$ has $K$-basis consisting of cosets of the form $x\otimes (m_q,\ldots, m_1)+C_q(M,KY)$ with $x\in X\setminus Y$, and hence maps isomorphically to $C_q(M,K[X\setminus Y])$ via the map sending $x\otimes (m_q,\ldots, m_1)+C_n(M,KY)$ to $(x+KY)\otimes (m_q,\ldots, m_1)$ for $x\in X\setminus Y$, as the elements $(x+KY)\otimes (m_q,\ldots, m_1)$ with $x\in X\setminus Y$ form a $K$-basis for $C_q(M,K[X\setminus Y])$.  The result follows from this and the naturality in (1), which then identifies $C_q(M,KX)/C_q(M,KY)$ with $C_q(\mathcal B(X\rtimes M),\mathcal B(Y\rtimes M),K)$.
\end{proof}

From the special case in which $X$ is a one-point set, and so $\mathcal B(X\rtimes M)=\mathcal BM$, we recover the classical fact that $H_n(M,K)\cong H_n(\mathcal BM,K)$ and $H^n(M,K)\cong H^n(\mathcal BM,K)$ for any commutative ring $K$.

\subsection{Group actions and reduction to posets}
An $M$-set $X$ is \emph{cyclic} if $X=xM$ for some $x\in X$; in particular, $X\neq \emptyset$.  To each $M$-set $X$, we have its associated poset $\Omega_M(X)$ of cyclic sub-$M$-sets.  So $\Omega_M(X) =\{xM\mid x\in X\}$ ordered by inclusion; if $M$ is understood, we just write $\Omega(X)$.  Note that $\Omega_M(U(M))$ is the poset $M/{\mathscr R}$ of principal right ideals.  More generally, if $R\subseteq M$ is a right ideal, then $\Omega(R)$ is the subposet of principal right ideals of $M$ contained in $R$.

For an $M$-set $X$, there is an evident functor $\Phi_X\colon X\rtimes M\to \Omega(X)$ sending $x$ to $xM$ on objects and sending $(x,m)\colon xm\to x$ to the arrow $(xmM,xM)$.

\begin{Rmk}\label{r:path.comp.bij}
The functor $\Phi_X\colon X\rtimes M\to \Omega(X)$ is surjective on objects and there is an arrow from $x$ to $y$ in $X\rtimes M$ if and only if $xM\subseteq yM$.  Hence, $\Phi_X$ induces a bijection of the sets of path components of the classifying spaces of these two categories, natural in $X$.
\end{Rmk}

The following was proved in~\cite[Theorem~6.15]{MSS} as an application of Quillen's Theorem~A~\cite{Quillen} (see~\cite{Nunes} for related results).  The key observation is that the category $\Phi_X/xM$ is isomorphic to $xM\rtimes M$.

\begin{Thm}\label{t:equaltoposet}
Suppose that $X$ is a right $M$-set such that each cyclic $M$-subset $xM$ with
$x\in X$ is contractible (i.e., $\mathcal B(xM\rtimes M)$ is contractible).  Then $\Phi_X\colon X\rtimes M\to \Omega(X)$ induces
a homotopy equivalence of classifying spaces $\mathcal B\Phi_X\colon \mathcal B(X\rtimes M)\to \Delta(\Omega(X))$.
\end{Thm}

A monoid $M$ is said to be \emph{right p.p.}~(principally projective) if each principal right ideal of $M$ is a projective right $M$-set or, equivalently, is isomorphic to an $M$-set of the form $eM$ with $e=e^2$.   Recall that an element $m$ of a monoid $M$ is \emph{(von Neumann) regular} if there exists $n\in M$ with $mnm=m$.  One says that $M$ is \emph{regular} if each element of $M$ is regular.  For a regular monoid $M$, each principal right ideal is generated by an idempotent and hence $M$ is right p.p.  Fountain proved that $M$ is right p.p.\ if and only if each $\mathscr L^*$-class of $M$ contains an idempotent, cf.~\cite{fountain}.  Recall that $m\eL^* m'$ if, for all $x,y\in M$, one has $mx=my\iff m'x=m'y$, or equivalently there is an $M$-set isomorphism $mM\to m'M$ taking $m$ to $m'$. Of course, left p.p.\ is defined dually.
The following consequence of Theorem~\ref{t:equaltoposet} and Corollary~\ref{c:contract.proj} is~\cite[Corollary~6.18]{MSS}.

\begin{Cor}\label{c:rightppcase}
Let $M$ be a right p.p.\ monoid (e.g., a regular monoid) and let $R$ be
a right ideal of $M$.  Then $\mathcal B\Phi_R\colon \mathcal B(R\rtimes M)\to \Delta(\Omega(R))$ is a homotopy equivalence.
\end{Cor}

Let $G$ be a group.  By a \emph{$G$-category} (\emph{$G$-poset}), we mean a category (poset) with an action of $G$ by automorphisms.  A functor $F\colon \mathscr C\to \mathscr D$ of $G$-categories is $G$-equivariant if it commutes with the $G$-actions.
The classifying space of a $G$-category (order complex of a $G$-poset) is easily checked to be a $G$-CW complex and, in fact, the classifying space construction yields a functor from the category of $G$-categories to the category of $G$-CW complexes with appropriate equivariant maps.   Recall here that a $G$-CW complex is a CW complex with an action of $G$ by cellular maps such that the $G$-action permutes the cells of each dimension and the setwise stabilizer of a cell is its pointwise stabilizer. As is usual, if $X$ is a set/space/module with a $G$-action, then $X^G$ will denote the $G$-invariants (or fixed points) of $X$.

If $X$ is a right $M$-set and $G=\mathrm{Aut}_M(X)$, then $X\rtimes M$ is a $G$-category where $G$ acts on objects via its original action and on arrows by $g(x,m) = (g(x),m)$.  Also $\Omega(X)$ is a $G$-poset via $g(xM) = g(x)M$.  It is evident that $\Phi_X$ is $G$-equivariant.  Also, $G\backslash X$ is an $M$-set with action $(Gx)m = G(xm)$.

\begin{Cor}\label{c:equivariant.business}
Let $M$ be a right p.p.~monoid (e.g., regular) and $R$ a right ideal of $M$.  Let $K$ be a commutative ring with unit.  Then $\Phi_R\colon R\rtimes M\to \Omega(R)$ induces isomorphisms  $H_n(\mathcal B(R\rtimes M),K)\to H_n(\Delta(\Omega(R)),K)$ and $H^n(\Delta(\Omega(R)),K)\to H^n(\mathcal B(R\rtimes M),K)$ of $K\mathrm{Aut}_M(R)$-modules.
\end{Cor}
\begin{proof}
Equivariance of $\Phi_R$ provides the desired $K\mathrm{Aut}_M(R)$-module homomorphisms.  These are isomorphisms  by Corollary~\ref{c:rightppcase} since $\mathcal B\Phi_R$ is a homotopy equivalence.
\end{proof}

Recall that a monoid $M$ is \emph{Dedekind-finite} if each left invertible element of $M$ is  right invertible.  This is easily checked to be equivalent to the nonunits of $M$ forming a two-sided ideal. Of course, every finite monoid is Dedekind-finite.   For a Dedekind-finite monoid $M$ with group of units $G$ and ideal of singular elements $S=M\setminus G$,  we have a surjective homomorphism $\rho\colon KM\to KG$ given by the identity map on $G$ and the zero map on $S$ (in fact, $KM/KS\cong KG$). Hence any (simple) $KG$-module can be inflated to a (simple) $KM$-module.  For the moment, we are interested in the case of the trivial $KG$-module.  To avoid notational confusion with the trivial $KM$-module, we shall write $K_{(G)}$ for the inflation of the trivial $KG$-module to a $KM$-module.

Retaining the above notation, we write $S\mathcal EM$ for $\mathcal B(S\rtimes M)$ as it consists of all simplices of $\mathscr EM$ of the form $s\sigma$ with $s\in S$.  Since $S$ is a two-sided ideal, $(U(M),S)$ is an $M$-set pair, and one has $M=\End_M(U(M),S)$ and $G=\mathrm{Aut}_M(U(M),S)$.  Hence $U(M)\rtimes M$ and $S\rtimes M$ are $G$-categories.  Note that taking $G$-orbits turns $(G\backslash U(M),G\backslash S)$ into an $M$-set pair.  It is immediate from the definitions that the CW pairs $(G\backslash \mathscr EM, G\backslash S\mathscr EM)$ and $(\mathcal B(G\backslash U(M)\rtimes M),\mathcal B(G\backslash S\rtimes M))$ are isomorphic.   Hence we have the following.

\begin{Thm}\label{t:two.trivials}
Let $M$ be a Dedekind-finite monoid with group of units $G$ and $K$ a commutative ring with unit.  Then \[\Ext^n_{KM}(K,K_{(G)})\cong H^n(G\backslash\mathscr EM,G\backslash S\mathscr EM,K)\] for all $n\geq 0$.
\end{Thm}
\begin{proof}
Put $X=G\backslash U(M)$ and $Y=G\backslash S$.  Then $X\setminus Y =\{G\}$ and $K^{X\setminus Y}\cong K_{(G)}$ via $f\mapsto f(G)$ because if $f\in K^{X\setminus Y}$ and $s\in S$, then $(sf)(G) = f(Gs) =0$ and if $g\in G$, then $(gf)(G) = f(Gg)=f(G)$.  The result then follows from Proposition~\ref{p:pass.to.relative}(3) and the discussion above.
\end{proof}

There is a dual result in the case that $M$ is finite and $K$ is a field, computing $\Ext$ from $K_{(G)}$ to $K$.  In order to state this, we need to recall some standard facts about finite dimensional algebras~\cite{assem,benson}.

\begin{Rmk}\label{r:duality}
If $A$ is a finite dimensional algebra over a field $K$, then the vector space dual functor $D$ gives a contravariant equivalence between the categories of finite dimensional left $A$-modules and finite dimensional left $A^{op}$-modules (which we shall view as right $A$-modules).  It follows that $\Ext^n_A(B,C)\cong \Ext^n_{A^{op}}(D(C),D(B))$ for any finite dimensional $A$-modules $B,C$.
\end{Rmk}

If $M$ is a monoid, then $(KM)^{op}\cong K[M^{op}]$.  Notice that  $G^{op}\cong G$ via inversion.    If $L$ is a left ideal of $M$, then $\Omega_{M^{op}}(L)$ is the poset of principal left ideals of $M$ contained in $L$.  Also note that $M$ acts naturally on the right of $\mathcal EM^{op}$ by cellular mappings (induced from the corresponding action on the nerve of $U(M)\rtimes M^{op}$ by simplicial mappings) and $\mathcal B(S\rtimes M^{op})=\mathcal EM^{op}S$, i.e., consists of those simplices of the form $\sigma s$ with $s\in S$.

\begin{Cor}\label{c:two.trivials.dual}
Let $M$ be a finite monoid with group of units $G$ and $K$ a field.  Then $\Ext^n_{KM}(K_{(G)},K)\cong H^n(\mathscr EM^{op}/G,\mathscr EM^{op}S/G,K)$ for all $n\geq 0$.
\end{Cor}
\begin{proof}
Observe that $D(K)$ is the trivial $KM^{op}$-module and $D(K_{(G)})$ is inflation of the trivial $KG^{op}$-module to $KM^{op}$.  The result then follows from Remark~\ref{r:duality} and Theorem~\ref{t:two.trivials}.
\end{proof}

The case $n=1$ of Theorem~\ref{t:two.trivials} and Corollary~\ref{c:two.trivials.dual} can be simplified  using the long exact sequence in cohomology and group cohomology to prove $\Ext^1_{KM}(K,K_{(G)})\cong \til H^0(\Delta(\Omega_M(S))),K)^G$ and $\Ext^1_{KM}(K_{(G)},K)\cong \til H^0(\Delta(\Omega_{M^{op}}(S)),K)^G$, generalizing a result
stated in~\cite[Remark~2B.12]{Khovanov}.  However, we shall deduce this later (Corollary~\ref{c:ext1.two.trivs}) from a more general result.

\subsection{Tor and homological epimorphisms}
The goal in this subsection is to give a topological characterization of when a homomorphism $\p\colon M\to N$ of monoids induces a homological epimorphism $\ov\p\colon KM\to KN$ of monoid algebras for a commutative ring $K$.     To do so, we generalize Proposition~\ref{p:pass.to.relative}(1) to a method to compute the functor $\mathrm{Tor}$ between $KX$ and $KY$ where $X$ is a right $M$-set and $Y$ is a left $M$-set.  This makes use of the classifying space of a certain category whose nerve is the corresponding two-sided bar construction. These results will only be used to prove  Corollary~\ref{c:homological.epi} and its consequence Corollary~\ref{c:ext.inflate.mn} and so can be omitted on a first reading.

Let $X$ be a right $M$-set and $Y$ a left $M$-set.  Define a category \mbox{$X\rtimes M\ltimes Y$} with object set $X\times Y$ and arrow set $X\times M\times Y$.  The arrow $(x,m,y)$ goes from $(xm,y)$ to $(x,my)$. We often draw the arrow $(x,m,y)$ as $(xm,y)\xrightarrow{\,m\,}(x,my)$.   The product of two composable arrows is given by \[(x,m,m'y )(xm ,m',y) = (x,mm',y).\] The identity arrow at $(x,y)$ is $(x,1,y)$.  This construction is functorial in the pair $(X,Y)$.  The classifying space $\mathcal B(X\rtimes M\ltimes Y)$ will be denoted $\mathcal B(X,M,Y)$ for brevity and to agree with standard notation for two-sided bar constructions.  The $0$-simplices of the nerve are the elements of $X\times Y$.  A $q$-simplex, for $q\geq 1$, is of the form
\begin{equation}\label{eq:twosidedbar}
\begin{split}
(xm_q\cdots m_1,y)\xrightarrow{m_1} (xm_q\cdots m_2,m_1y)\xrightarrow{m_2} \cdots\\ \qquad \xrightarrow{m_{q-1}} (xm_q,m_{q-1}\cdots m_1y) \xrightarrow{m_q} (x,m_q\cdots m_1y).
\end{split}
\end{equation}
The $q$-simplex in \eqref{eq:twosidedbar} will be denoted $(x,m_q,\ldots, m_1,y)$ for convenience, and hence we may identity the set of $q$-simplices with $X\times M^q\times Y$ for $q\geq 0$. A $q$-simplex $(x,m_q,\ldots, m_1,y)$ is degenerate if and only if $m_i=1$ for some $i$.  If $K$ is a commutative ring, then the boundary map $d\colon C_q(\mathcal B(X,M,Y),K)\to C_{q-1}(\mathcal B(X,M,Y),K)$ for $q\geq 1$ is given by
\begin{equation}\label{eq:boundary.twosided}
\begin{split}
d(x,m_q,\ldots, m_1,y) &= (x,m_q,\ldots, m_2, m_1y)+\sum_{i=1}^{q-1}(-1)^i(x,m_q,\ldots, m_{i+1}m_i,\ldots, m_1, y)\\ &\quad  +(-1)^q(xm_q,m_{q-1},\ldots,m_1,y)
\end{split}
\end{equation}
where, as usual, a degenerate simplex is identified with $0$.

Classical two-sided bar resolutions imply that \[\mathrm{Tor}^{KM}_q(KX,KY)\cong H_n(\mathcal B(X,M,Y),K)\] for $q\geq 0$.  Let us give a short topological proof.  Let $Y$ be a left $M$-set and again denote by $U(M)$ the underlying set of $M$, which is a right $M$-set via multiplication.   Note that $M$ has a left action on  $U(M)\rtimes M\ltimes Y$ via $m'(m,y) = (m'm,y)$ on objects and $m'(m_0,m,y) = (m'm_0,m,y)$ on arrows.  Therefore, we have that $C_\bullet(\mathcal B(U(M),M,Y),K)$  is a chain complex of $KM$-modules.  Moreover, $C_q(\mathcal B(U(M),M,Y),K)$ has $K$-basis $M\times (M\setminus \{1\})^q\times Y$, and is, in fact, a free $KM$-module with basis $\{1\}\times (M\setminus \{1\})^q\times Y$ (which we identify with $(M\setminus \{1\})^q\times Y$).  From now on we write $m(m_q,\ldots, m_1,y)$ for the $q$-simplex $(m,m_q,\ldots, m_1,y)$.  It follows from \eqref{eq:boundary.twosided} that on the $KM$-basis $(M\setminus \{1\})^q\times Y$ for $C_q(\mathcal B(U(M),M,U(M)),K)$, we have that
\begin{equation}\label{eq:boundary.twosided.bar.regular}
\begin{split}
d_q(m_q,\ldots, m_1,y) &= (m_q,\ldots, m_2,m_1y)+\sum_{i=1}^{q-1}(-1)^i(m_q,\ldots, m_{i+1}m_i,\ldots, m_1,y)\\ &\quad +(-1)^qm_q(m_{q-1},\ldots,m_1,y)
\end{split}
\end{equation}
with the usual convention that degenerate simplices are treated as $0$.

\begin{Prop}\label{p:free.res}
Let $M$ be a monoid,  $Y$ a left $M$-set and $K$ a commutative ring.  Then the augmented chain complex \[C_\bullet(\mathcal B(U(M),M,Y),K)\to KY\] is a free resolution of $KY$ as a $KM$-module, where the augmentation \[\epsilon\colon C_0(\mathcal B(U(M),M,Y),K)\to KY\] is given by $\epsilon(m,y)= my$ for $(m,y)\in U(M)\times Y$.
\end{Prop}
\begin{proof}
Clearly $\epsilon$ is $M$-equivariant. Fix $y\in Y$ and let $\mathscr C_y$ be the full subcategory of $U(M)\rtimes M\ltimes Y$ on the objects $(m,y')$ with $my'=y$.  Note that an arrow $(m_0,m,y)$ belongs to $\mathscr C_{m_0my}$, and so $U(M)\rtimes M\ltimes Y=\coprod_{y\in Y}\mathscr C_y$.  Moreover, if $(m,y')$ is an object of $\mathscr C_y$, then $\epsilon(m,y')=y$.  Thus the restriction of $\epsilon$ to $C_0(\mathcal BC_y,K)$ is the standard augmentation map from topology, multiplied  on the right by $y$.  Therefore, it suffices to show that $\mathcal B\mathscr C_y$ is contractible for all $y\in Y$. But $\mathscr C_y$ has terminal object $(1,y)$ since if $my'=y$, then $(1,m,y')$ is the unique arrow $(m,y')\to (1,y)$.    Thus $\mathcal B\mathscr C_y$ is contractible, completing the proof.
\end{proof}

We may now prove a generalization of Proposition~\ref{p:pass.to.relative} (which is the special case where $Y$ is a one-point set).  The proof is nearly identical, and so we shall be brief.

\begin{Thm}\label{t:compute.tor}
Let $M$ be a monoid and $K$ a commutative ring. Let $X$ be a right $M$-set and $Y$ a left $M$-set.
\begin{enumerate}
\item  $\Tor^{KM}_i(KX,KY)\cong H_i(\mathcal B(X,M,Y),K)$ for all $i\geq 0$.
\item  $\Ext_{KM}^i(KY,K^X)\cong H^i(\mathcal B(X,M,Y),K)$ for all $i\geq 0$.
\end{enumerate}
\end{Thm}
\begin{proof}
We begin with the first item.
Note $KX\otimes_{KM} C_q(\mathcal B(U(M),M,Y),K)$ is a free $K$-module with basis the elementary tensors $x\otimes (m_q,\ldots, m_1,y)$ with $x\in X$, $y\in Y$ and $m_i\in M\setminus \{1\}$ for $i=1,\ldots, q$.  This is isomorphic as a $K$-module to $C_q(\mathcal B(X,M,Y),K)$ via $x\otimes (m_q,\ldots, m_1,y)\mapsto (x,m_q,\ldots, m_1,y)$.  A comparison of \eqref{eq:boundary.twosided} and \eqref{eq:boundary.twosided.bar.regular} shows these maps yield an isomorphism of chain complexes.  Since $C_q(\mathcal B(U(M),M,Y),K)\to KY$ is a free resolution of $KY$ as a $KM$-module thanks to  Proposition~\ref{p:free.res}, we deduce that $\Tor^{KM}_i(KX,KY)\cong H_i(\mathcal B(X,M,Y),K)$ for all $i\geq 0$.

For the second item, note that $\Hom_{KM}(C_q(B(UM,M,Y),K),K^X)\cong (K^X)^{(M\setminus \{1\})^q\times Y}\cong K^{X\times (M\setminus \{1\})^q\times Y}\cong C^q(\mathcal B(X,M,Y),K)$ via the currying isomorphism.  The proof that this yields an isomorphism of cochain complexes is straightforward from \eqref{eq:boundary.twosided} and \eqref{eq:boundary.twosided.bar.regular}, and is nearly identical to the proof of Proposition~\ref{p:pass.to.relative}(2), and so we omit it.  As $C_q(\mathcal B(U(M),M,Y),K)\to KY$ is a free resolution of $KY$ as a $KM$-module by Proposition~\ref{p:free.res}, we conclude that $\Ext_{KM}^i(KY, K^X)\cong H^i(\mathcal B(X,M,Y),K)$ for all $i\geq 0$.
\end{proof}

Recall that an \emph{epimorphism} in a category $\mathscr C$ is an arrow $f$ such that $gf=hf$ implies that $g=h$, that is, $f$ is right cancellable.  Every surjective homomorphism of monoids (respectively, rings) is an epimorphism but in both these categories there are epimorphisms that are not surjective.  By Isbell's zig-zag theorem, a ring homomorphism $\p\colon R\to S$ is an epimorphism if and only if the multiplication map $S\otimes_R S\to S$ is an $S$-bimodule isomorphism.  There is an analogous result in the monoid context, also due to Isbell.

A ring homomorphism $\p\colon R\to S$ is called a \emph{homological epimorphism}~\cite{homologicalepi} if it is an epimorphism and if $\Tor^R_i(S,S)=0$ for all $i\geq 1$.  Note that $\Tor^R_0(S,S)\cong S\otimes _R S\cong S$ since $\p$ is an epimorphism.   The following theorem is part of~\cite[Theorem~4.4]{homologicalepi} and explains the importance of homological epimorphisms.

\begin{Thm}\label{t:homological.epi}
Let $\p\colon R\to S$ be a ring homomorphism.  Then the following are equivalent.
\begin{enumerate}
\item $\p$ is a homological epimorphism.
\item For all right $S$-modules $U$ and left $S$-modules $V$, the natural map $\Tor^R_i(U,V)\to \Tor^S_i(U,V)$ is an isomorphism for all $i\geq 0$.
\item For all left $S$-modules $U,V$, the natural map $\Ext_S^i(U,V)\to \Ext_R^i(U,V)$ is an isomorphism for all $i\geq 0$.
\item For all right $S$-modules $U,V$, the natural map $\Ext_S^i(U,V)\to \Ext_R^i(U,V)$ is an isomorphism for all $i\geq 0$.
\end{enumerate}
\end{Thm}

We can now give a topological characterization of when a monoid homomorphism induces a homological epimorphism.  This is analogous to a result of Nunes~\cite{Nunes}, who was interested in monoid cohomology and therefore used a one-sided bar construction.

Let $\p\colon M\to N$ be a homomorphism of monoids.  Then we can view $N$ as both a right and left $M$-set via $\p$.  The category $N\rtimes M\ltimes N$ was considered by Rhodes and Tilson~\cite{Kernel} as a substitute for the kernel of a monoid homomorphism and is called the kernel category of $\p$ (actually they call a certain quotient of this category the kernel category).  In~\cite{Cats2} this category is denoted $\ker(\p)$, and we shall use that notation here. Note that for a group homomorphism $\p\colon G\to H$, the category $\ker (\p)$ is a groupoid equivalent to the disjoint union of $|H|$ of copies of the group theoretic kernel of $\p$ (which we denote by $\ker \p$ without parentheses).

\begin{Thm}\label{t:homological.epi.monoids}
Let $\p\colon M\to N$ be a monoid homomorphism and $K$ a commutative ring.  Then the extension $\ov\p\colon KM\to KN$ is a homological epimorphism if and only if $\p$ is a monoid epimorphism and $H_i(\mathcal B(N,M,N),K)=0$ for all $i\geq 1$.
\end{Thm}
\begin{proof}
In light of Theorem~\ref{t:compute.tor}, it suffices to show that $\p$ is an epimorphism if and only if $\ov \p$ is one.  Trivially, if $\ov\p$ is an epimorphism and $\alpha,\beta\colon N\to N'$ are monoid homomorphisms with $\alpha\p=\beta\p$, then $\ov\alpha\ov\p=\ov \beta\ov \p$ and so $\ov \alpha=\ov\beta$.  But this implies $\alpha=\beta$.  Conversely, suppose that $\p$ is a monoid epimorphism and $\alpha,\beta\colon KN\to R$ are ring homomorphisms such that $\alpha \ov \p=\beta\ov \p$.  Let $E=\{a\in KN\mid \alpha(a)=\beta(a)\}$ be the equalizer of $\alpha$ and $\beta$.  Then $E$ is a subring of $KN$ and we aim to show that $E=KN$.  By hypothesis, $K\cdot 1\subseteq \ov \p(KM)\subseteq E$, and so $E$ is a $K$-subalgebra.  It therefore, suffices to show that $N\subseteq E$.  Since $\p$ is an epimorphism and $\alpha|_N\p = (\alpha\ov \p)|_M=(\beta \ov \p)|_M = \beta|_N\p$, we deduce that $\alpha|_N=\beta|_N$, that is, $N\subseteq E$, as required.  This completes the proof.
\end{proof}

\subsection{Applications to semidirect and crossed products}

Using an idea from~\cite{Kernel}, we show that if $N$ is a monoid acting on the left of a monoid $M$ by endomorphisms and $\pi\colon M\rtimes N\to N$ is the semidirect product projection, then the induced homomorphism $\ov \pi\colon K[M\rtimes N]\to KN$ is a homological epimorphism if and only if the classifying space $\mathcal BM$ of $M$ is $K$-acyclic.  This should be compared with~\cite{Nunes} where an analogue for monoid cohomology was proved (and which is a consequence of our result).  It turns out that we can work with a more general construction than the semidirect product, which is the monoid analogue of the group crossed product~\cite[Chapter IV.8]{MacHomology}.  I presume this is folklore, but the only explicit reference that I  found was~\cite{crossedproduct}, which gives few details.

By a \emph{normalized crossed system} over the pair of monoids $M,N$, we mean a $4$-tuple $(\alpha, c,M,N)$ where $\alpha\colon N\to \End(M)$ and $c\colon N\times N\to M$ are mappings such that for all $m\in M$ and $n,n_1,n_2,n_3\in N$:
\begin{itemize}
\item [(C1)] $\alpha_{n_1}(\alpha_{n_2}(m))c(n_1,n_2) = c(n_1,n_2)\alpha_{n_1n_2}(m)$;
\item [(C2)] $c(n_1,n_2)c(n_1n_2,n_3) = \alpha_{n_1}(c(n_2,n_3))c(n_1,n_2n_3)$;
\item [(C3)] $\alpha_1(m)=m$ (i.e., $\alpha_1=1_M$);
\item [(C4)] $c(1,n)=1=c(n,1)$;
\end{itemize}
where we write $\alpha_n$ for $\alpha(n)$.
One calls $\alpha$ a weak action and $c$ an $\alpha$-cocycle.  Note that if $c$ is trivial (i.e., $c(n_1,n_2)=1$ for all $n_1,n_2$), then (C1) and (C3) imply that $\alpha$ is a homomorphism.  However, we do not in general require $\alpha$ to be a homomorphism.

Let $G$ be the group of units of $M$.  Then $G^{N\setminus \{1\}}$ acts on the set of normalized crossed systems with respect to $M,N$. If $f\in G^{N\setminus \{1\}}$, then $f\cdot (\alpha,c,M,N)= (\alpha',c',M,N)$ is defined as follows.  Put $f(1)=1$ for convenience.   Then $\alpha'_n(m) = f(n)\alpha_n(m)f(n)\inv$ and \[c'(n_1,n_2) = f(n_1)\alpha_{n_1}(f(n_2))c(n_1,n_2)f(n_1n_2)\inv.\]  We say that two crossed systems are cohomologous if they are in the same orbit of this action.  Notice that if $M$ is an abelian group, then (C1) and (C3) imply that $\alpha$ is a homomorphism, and so one can identify an equivalence class of crossed systems over $M,N$ with a $\mathbb ZN$-module structure on $M$ together with a cohomology class in $H^2(N,M)$ with respect to this module structure on $M$ (as $2$-cocycles in the normalized cochain complex for $M$ are maps satisfying (C2) and (C4), and the two notions of being cohomologous coincide).

Let $(\alpha,c,M,N)$ be a normalized crossed system.  Then we can define the \emph{crossed product} $M\rtimes_{\alpha,c} N$ to be the monoid with underlying set $M\times N$, product
\[(m_1,n_1)(m_2,n_2) = (m_1\alpha_{n_1}(m_2)c(n_1,n_2), n_1n_2)\]
and identity element $(1,1)$.  It is straightforward to verify that conditions (C1)--(C4) are equivalent to this product being associative with identity $(1,1)$.  In the case that $\alpha$ is a homomorphism and $c$ is trivial, we obtain the semidirect product $M\rtimes_{\alpha} N$ (or simply $M\rtimes N$ if $\alpha$ is understood).  Note that $\pi\colon M\rtimes_{\alpha,c} N\to N$ given by $\pi(m,n)=n$ is a homomorphism and $\pi\inv(1) = M\times \{1\}\cong M$.

\begin{Prop}\label{p:equiv.crossed}
If $(\alpha,c,M,N)$ and $(\alpha',c',M,N)$ are equivalent normalized crossed systems, then $M\rtimes_{\alpha,c} N\cong M\rtimes_{\alpha',c'} N$ via an isomorphism commuting with the projection to $N$.
\end{Prop}
\begin{proof}
Suppose that $(\alpha',c',M,N) = f\cdot (\alpha,c,M,N)$ with $f\in G^{N\setminus \{1\}}$ and extend $f$ to $N$ by $f(1)=1$.  Then it is a routine computation to verify that $M\rtimes_{\alpha',c'} N\to M\rtimes_{\alpha,c} N$ given by $(m,n)\mapsto (mf(n),n)$ is an isomorphism.
\end{proof}

Crossed products also enjoy an internal characterization.
  If $(\alpha,c,M,N)$ is a normalized crossed system, then $s\colon N\to M\rtimes_{\alpha,c}N$ given by $s(n) =(1,n)$ is a set-theoretic section to the projection \[\pi\colon M\rtimes_{\alpha,c} N\to N\] with $s(1)=(1,1)$, $\pi\inv(1)\cong M$ and $s(N)$ a basis for the left $M$-set $M\rtimes_{\alpha,c} N$ (with action $m(m_1,n_1) = (mm_1,n_1)$).  Note that $s$ is a homomorphism if and only if $c$ is trivial (i.e., $M\rtimes_{\alpha,c} N$ is a semidirect product).  Conversely, such a section gives rise to a crossed product decomposition.

\begin{Prop}\label{p:crossed.abstract}
Let $\p\colon L\to N$ be a surjective homomorphism of monoids and put $M=\pinv(1)$.  Suppose that there is a set-theoretic section $s\colon N\to L$ of $\p$ such that $s(N)$ is a basis for $L$ as a left $M$-set and $s(1)=1$.  Then there is a normalized crossed system $(\alpha, c,M,N)$ with an isomorphism $\gamma\colon M\rtimes_{\alpha,c} N\to L$  given by $\gamma(m,n) = ms(n)$ over $N$.  Replacing $s$ by a different section with the same properties results in an equivalent normalized crossed system.  Moreover, if $s$ is a homomorphism, then $c$ can be taken to be trivial, and hence $L$ is isomorphic to a semidirect product of $M$ and $N$.
\end{Prop}
\begin{proof}
We only sketch the straightforward argument.  As usual, $G$ will denote the group of units of $M$.  First note that we have $\pinv(n) = Ms(n)$ since if $a\in L$, then $a=ms(n)$ for some $m\in M$, $n\in N$, and  $\p(a) =\p(m)\p(s(n)) = n$.   In particular, if $m\in M$ and $n\in N$, then $\p(s(n)m) = n$, and so $s(n)m =\alpha_n(m)s(n)$ for a unique element $\alpha_n(m)\in M$.  We claim that $\alpha_n\in \End(M)$.  Indeed, $\alpha_n(mm')s(n) = s(n)(mm') = (s(n)m)m' = \alpha_n(m)(s(n)m') = \alpha_n(m)\alpha_n(m')s(n)$, whence $\alpha_n(mm')=\alpha_n(m)\alpha_n(m')$.   Also, $s(n)1=1s(n)$ implies $\alpha_n(1)=1$.  Trivially, $\alpha_1=1_M$ as $s(1)=1$.   Define $c\colon N\times N\to M$ by $s(n_1)s(n_2) = c(n_1,n_2)s(n_1n_2)$ (using that $\p(s(n_1)s(n_2))=n_1n_2$); notice that $c$ is trivial if $s$ is a homomorphism.  Axioms (C1)--(C4) follow easily from associativity in $L$ and from $s(1)=1$.  One can immediately check that $\gamma$ as defined above is an isomorphism.

If $s'\colon N\to L$ is another section with $s'(1)=1$ and $s'(N)$ a basis for $L$ as a left $M$-set, then for each $n\in N\setminus \{1\}$, we can find $a_n,b_n\in M$ with $a_ns(n)=s'(n)$ and $b_ns'(n)=s(n)$ as $\rho(s(n))=n=\rho(s'(n))$.  Then $a_nb_ns'(n)=s'(n)$, and so $a_nb_n=1$ by definition of a basis and, similarly, $b_na_n=1$. Thus $a_n\in G$ for all $n\in N\setminus \{1\}$.  Define $f\in G^{N\setminus \{1\}}$ by $f(n) =a_n$.  It is routine to verify that the crossed system associated to $s'$ is $f\cdot (\alpha,c,M,N)$.
\end{proof}

Note that every surjective group homomorphism $\p\colon G\to H$ satisfies the conditions of Proposition~\ref{p:crossed.abstract}, and so $G\cong \ker \p\rtimes_{\alpha,c} H$ for an appropriate choice of $\alpha,c$.  This is of course well known~\cite[Chapter IV.8]{MacHomology}.

The main result of this subsection determines when the projection from a crossed product (and, in particular, semidirect product) induces a homological epimorphism of monoid algebras.

\begin{Thm}\label{t:semidirect}
Let $(\alpha,c,M,N)$ be a normalized crossed system and let $\pi\colon M\rtimes_{\alpha,c} N\to N$ be the projection.   Then $\mathcal B(N,M\rtimes_{\alpha,c} N,N)$ is homotopy equivalent to the topological sum $\coprod_{N} \mathcal BM$ of $|N|$ copies of $\mathcal BM$.  Consequently, if $K$ is a commutative ring, then $\ov \pi\colon K[M\rtimes_{\alpha,c} N]\to KN$ is a homological epimorphism if and only if the classifying space $\mathcal BM$ is $K$-acyclic.  In fact, $\Tor^{K[M\rtimes_{\alpha,c} N]}_i(KN,KN)\cong \bigoplus_N H_i(\mathcal BM,K)\cong  \bigoplus_{N} H_i(M,K)$ for $i\geq 0$.  This applies, in particular, to a semidirect product projection.
\end{Thm}
\begin{proof}
Let $n\in N$ and let $\mathscr C_n$ be the full subcategory of $\ker(\pi)=N\rtimes (M\rtimes_{\alpha,c} N)\ltimes N$ on the objects $(a,b)\in N\times N$ with $ab=n$.  An arrow of $\ker(\pi)$ is of the form $(a,(m,n),b)\colon (an,b)\to (a,nb)$, and this arrow belongs to $\mathscr C_{anb}$.  Thus $\ker(\pi)=\coprod_{n\in N}\mathscr C_n$, and so it suffices to show that $\mathcal B\mathscr C_n$ is homotopy equivalent to $\mathcal BM$ for all $n\in N$.  We prove this using Quillen's Theorem~A applied to a certain functor $\ker(\pi)\to M$, which generalizes one considered by Rhodes and Tilson~\cite{Kernel} for semidirect products.

Fix $n_0\in N$.  Define $F\colon \mathscr C_{n_0}\to M$ by sending each object of $\mathscr C_{n_0}$ to the unique object of $M$ and putting $F(a,(m,n),b)= \alpha_a(m)c(a,n)$.  We check that $F$ is a functor. The identity at the object $(a,b)$ is $(a,(1,1),b)$, and $F(a,(1,1),b)=\alpha_a(1)c(a,1)=1$.   A pair of composable arrows of $\mathscr C_{n_0}$ is of the form $(a,(m,n),n'b),(an ,(m',n'),b)$ with $ann'b=n_0$, and the product of these arrows is $(a,(m\alpha_n(m')c(n,n'),nn'),b)$.  Therefore,
\begin{equation}\label{eq:crossed.ker}
\begin{split}
F((a,(m,n),n'b)(an ,(m',n'),b)) &= \alpha_a(m\alpha_n(m')c(n,n'))c(a,nn')\\ & = \alpha_a(m)\alpha_a\alpha_n(m')\alpha_a(c(n,n'))c(a,nn').
\end{split}
\end{equation}
But we also compute
\begin{equation}\label{eq:crossed.ker.2}
\begin{split}
 F(a,(m,n),n'b)F(an,(m',n'),b) &= \alpha_a(m)c(a,n)\alpha_{an}(m')c(an,n')\\ & = \alpha_a(m)\alpha_a\alpha_n(m')c(a,n)c(an,n')
 \end{split}
\end{equation}
using (C1).  An application of (C2) shows that right hand sides of \eqref{eq:crossed.ker} and \eqref{eq:crossed.ker.2} are equal.  Thus $F$ is a functor.

Denoting the unique object of $M$ by $\ast$, we claim that  $F/\ast$  has terminal object $((1,n_0),1)$.  An object of $F/\ast$ is of the form $((a,b),m)$ with $ab=n_0$ and $m\in M$.    But using that $\alpha_1(m')=m'$ and $c(1,a)=1$, it is straightforward to verify that  $(1,(m,a),b)$ provides the unique arrow of $F/\ast$ from $((a,b),m)$ to $((1,n_0),1)$.   Thus $\mathcal B(F/\ast)$ is contractible, and hence $\mathcal BF\colon \mathcal B\mathscr C_{n_0}\to \mathcal BM$ is a homotopy equivalence by Quillen's Theorem~A, as required.
\end{proof}

\section{Groups associated to monoids}\label{s:sec4}
The forgetful functor from the category of groups to the category of monoids admits both left and right adjoints.  The right adjoint takes a monoid $M$ to its group of units and the left adjoint takes $M$ to its \emph{group completion} $G(M)$.  As a group, $G(M)$ has generators $M$ and relations coming from the multiplication table of $M$. It is well known that $G(M)\cong \pi_1(\mathcal BM)$~\cite{GabrielZisman}.  If $M$ is a finite monoid, then $G(M)$ is a homomorphic image of $M$.  Indeed, the image of $M$ in $G(M)$ is a finite submonoid, and hence a subgroup, generating $G(M)$.  The following concrete description is a combination of~\cite[Chapter~8, Example~1.7]{Arbib} and Graham's theorem~\cite{Graham} (see also~\cite[Chapter~4.13]{qtheor}).    Choose an idempotent $e$ in the minimal ideal of $M$.  It is well known that $eMe=G_e$ is a group (cf.~\cite[Appendix~A]{qtheor} or~\cite{Arbib}).   Then $G(M)\cong G_e/N$ where $N$ is the normal closure of the subgroup of $G_e$ consisting of those elements than can be written as a product of idempotents from $MeM$.  The universal map takes $m$ to $emeN$.

 The canonical homomorphism $\psi\colon M\to G(M)$ allows us to inflate $KG(M)$-modules to $KM$.  A $KM$-module is inflated from $KG(M)$ if and only if the corresponding representation of $M$ is by invertible $K$-linear operators.

There are also groups associated to each idempotent of $M$.  The set of idempotents of a monoid $M$ will be denoted $E(M)$.
If $e\in E(M)$, then the right ideal $eM$ is a right $M$-set with endomorphism monoid $eMe$ (acting by left multiplication) (cf.~\cite[Proposition~1.8]{repbook}).  The automorphism group of $eM$ is then the group of units $G_e$ of $eMe$, which is called the \emph{maximal subgroup} of $M$ at $e$.  Note that $G_1$ is the group of units of $M$.  We denote by $R_e$ the set of elements $m\in M$ with $mM=eM$.  Then $R(e)= eM\setminus R_e$ is a right ideal of $M$ contained in $eM$ (possibly empty), and so $(eM,R(e))$ is an $M$-set pair.  Moreover, $R_e$ is $G_e$-invariant and a free left $G_e$-set (see~\cite[Proposition~1.10]{repbook}), and so $R(e)$ is $G_e$-invariant.  Thus $G_e\cong \mathrm{Aut}_M(eM,R(e))$ via its action by left multiplication.  Since $R_e=eM\setminus R(e)$, following our previous convention we write $KR_e$ for the $KG_e$-$KM$-bimodule $KeM/KR(e)$.   The right $KM$-module $KR_e$ is often called the right \emph{Sch\"utzenberger representation} of $M$ on $R_e$.  Dually, if we put $L_e=\{m\in M\mid Mm=Me\}$ and $L(e)=Me\setminus L_e$, then $L(e)$ is a left ideal, $G_e$ acts freely on the right of $L_e$ and $KL_e=KMe/KL(e)$ is a $KM$-$KG_e$-bimodule.

\begin{Prop}\label{p:go.to.sing}
Let $M$ be a monoid,   $e\in E(M)$ and $K$ a commutative ring.  Viewing $KR_e= KeM/KR(e)$ as a $KG_e$-$KM$-bimodule, we have that $H_n(M,KR_e)\cong \til H_{n-1}(\mathcal B(R(e)\rtimes M), K)$ as left $KG_e$-modules.
\end{Prop}
\begin{proof}
Note that $KR_e$ is the right $KM$-module associated to the $M$-set pair $(eM,R(e))$.  Moreover,  $G_e$ is the automorphism group of this $M$-set pair via its action by left multiplication.  It follows that $H_n(M,KR_e)\cong H_n(\mathcal B(eM\rtimes M), \mathcal B(R(e)\rtimes M),K)$ as $KG_e$-modules by Proposition~\ref{p:pass.to.relative}.  Since $\mathcal B(eM\rtimes M)$ is contractible by Corollary~\ref{c:contract.proj}, we obtain from the long exact sequence for reduced homology (and naturality of the boundary map in that long exact sequence) that $H_n(\mathcal B(eM\rtimes M), \mathcal B(R(e)\rtimes M),K)\cong \til H_{n-1}(\mathcal B(R(e)\rtimes M),K)$  as left $KG_e$-modules.
 This completes the proof.
\end{proof}

If $e\in E(M)$ is an idempotent and $V$ is a $KG_e$-module, then the \emph{coinduced} $KM$-module (see~\cite[Chapter~5]{repbook}) is $\Coind_e(V)=\Hom_{KG_e}(KR_e,V)$.  The functor $\Coind_e$ is an exact functor from $KG_e$-modules to $KM$-modules because $KR_e$ is a $KG_e$-$KM$-bimodule which is a free as a left $KG_e$-module. There is also an \emph{induced} $KM$-module $\Ind_e(V)=KL_e\otimes_{KG_e} V$ associated to $V$.  The functor $\Ind_e$ is an exact functor as $KL_e$ is a free right $KG_e$-module.    Observe that if $V$ is a right $KG_e$-module, then $\Hom_{KG_e}(KL_e,V)$  is a right $KM$-module that can be viewed as $\Coind_e(V)$ for the opposite monoid $M^{op}$.  From this viewpoint, there is the following duality relating induction and coinduction.

\begin{Prop}\label{p:dual.induced}
Let $M$ be a monoid, $e\in E(M)$ and $K$ a commutative ring.  Then $\Hom_K(\Ind_e(V),K)\cong \Hom_{KG_e}(KL_e,\Hom_K(V,K))$ as right $KM$-modules.
\end{Prop}
\begin{proof}
By the hom-tensor adjunction, \[\Hom_K(KL_e\otimes_{KG_e} V,K)\cong \Hom_{KG_e}(KL_e,\Hom_K(V,K))\] as $K$-modules.  Naturality of the hom-tensor adjunction implies it is an isomorphism of right $KM$-modules.
\end{proof}

We now compute the cohomology of a coinduced module.

\begin{Thm}\label{t:cohomology.coinduced}
Let $M$ be a monoid,  $e\in E(M)$  and $K$ a commutative ring. Let $V$ be a
$KG_e$-module.
\begin{enumerate}
  \item There is a spectral sequence $\Ext^p_{KG_e}(\til H_{q-1}(\mathcal B(R(e)\times M),K),V)\Rightarrow_p H^n(M,\Coind_e(V))$.
  \item If $M$ is right p.p.\ (e.g., regular), there is a spectral sequence \[\Ext^p_{KG_e}(\til H_{q-1}(\Delta(\Omega(R(e))),K),V)\Rightarrow_p H^n(M,\Coind_e(V)).\]
\end{enumerate}
Moreover, if $V$ is injective, then
\begin{enumerate}[resume]
  \item $H^n(M,\Coind_e(V))\cong \Hom_{KG_e}(\til H_{n-1}(\mathcal B(R(e)\times M),K),V)$ for all $n\geq 0$.
  \item If $M$ is right p.p.\ (e.g., regular), then \[H^n(M,\Coind_e(V))\cong \Hom_{KG_e}(\til H_{n-1}(\Delta(\Omega(R(e))),K),V)\] for all $n\geq 0$.
\end{enumerate}
\end{Thm}
\begin{proof}
Recall that $H^n(M,\Coind_e(V)) = \Ext^n_{KM}(K,\Hom_{KG_e}(KR_e,V))$.   Then since $KR_e$ is a free left $KG_e$-module, it follows  $KR_e\otimes_{KM} P$ is $KG_e$-projective for any projective $KM$-module $P$, as $\Hom_{KG_e}(KR_e\otimes_{KM} P,-)\cong \Hom_{KM}(P,\Hom_{KG_e}(KR_e,-))$.  Therefore,
Theorem~\ref{t:ext.dual.tor} applies to provide a spectral sequence \[\Ext^p_{KG_e}(\Tor_q^{KM}(KR_e,K),V)\Rightarrow_p \Ext^n_{KM}(K,\Hom_{KG_e}(KR_e,V)).\]  As $\Tor_q^{KM}(KR_e,K)=H_q(M,KR_e)$,   Proposition~\ref{p:go.to.sing} yields the spectral sequence in (1).  The second item follows from the first and Corollary~\ref{c:equivariant.business}, as $G_e$ is a subgroup of the automorphism group (as a right $M$-set) of the right ideal $R(e)$.  The final two items follow because the spectral sequences (1),(2) collapse on the $E_2$-page on the line $p=0$ when $V$ is injective (or using the conclusion of Theorem~\ref{t:ext.dual.tor} for the case that $C$ is injective, as $\Tor_n^{KM}(KR_e,K)=H_n(M,KR_e)\cong\til H_{n-1}(\mathcal B(R(e)\times M),K)$).
\end{proof}

Our goal now is to compute $\Ext$ from a $KG(M)$-module, inflated to $M$,  to a coinduced module.

If $V,W$ are $KM$-modules, then $V\otimes_K W$ is a $KM$-module via $m(v\otimes w)=mv\otimes mw$.  If $V$ is a $KG(M)$-module, then there is the contragredient module $V^*=\Hom_K(V,K)$ where $(gf)(v) = f(g\inv v)$ for $v\in V$ and $g\in G(M)$.
If $V$ is a $KG(M)$-module and $W$ is a $KM$-module, then $\Hom_K(V,W)$ is a $KM$-module via the action $(mf)(v)= mf(\psi(m)\inv v)$.

\begin{Thm}\label{t:to.triv}
Let $M$ be a monoid and $K$ a commutative ring.
Let $U,W$ be $KM$-modules and $V$ a $KG(M)$-module.  Then there is an isomorphism $\Hom_{KM}(U\otimes_K V,W)\cong \Hom_{KM}(U,\Hom_K(V,W))$ natural in $U,V,W$.  Moreover, if $V$ is a projective $K$-module, then \[\Ext^n_{KM}(U\otimes_K V,W)\cong \Ext^n_{KM}(U,\Hom_K(V,W))\] for all $n\geq 0$.  In particular, if $V$ is a projective $K$-module and $W$ (respectively, $U$) is an injective (respectively, projective) $KM$-module, then $\Hom_K(V,W)$ (respectively, $U\otimes_K V$) is an injective (respectively, projective) $KM$-module.
\end{Thm}
\begin{proof}
It suffices by the standard hom-tensor adjunction for $K$ to show that $\p\colon U\otimes_K V\to W$ is a $KM$-module homomorphism if and only if $\Phi\colon U\to \Hom_K(V,W)$ given by $\Phi(u)(v) = \p(u\otimes v)$ is a $KM$-module homomorphism.  If $\p$ is a $KM$-module homomorphism, then
\begin{align*}
(m\Phi(u))(v) &= m\Phi(u)(\psi\inv(m)v)=m \p(u\otimes \psi\inv(m)v)\\ &=\p(mu\otimes \psi(m)\psi(m)\inv v)=\p(mu\otimes v) = \Phi(mu)(v).
\end{align*}
  Conversely, if $\Phi$ is a $KM$-module homomorphism, then
\begin{align*}
\p(m(u\otimes v)) &= \p(mu\otimes \psi(m)v) = \Phi(mu)(\psi(m)v)=(m\Phi(u))(\psi(m)v)\\ &=m\Phi(u)(\psi\inv(m) \psi(m)v) = m\p(u\otimes v)
\end{align*}
 as required.

If $V$ is a projective $K$-module, then $-\otimes_K V$ and $\Hom_K(V,-)$ are exact functors, and so by the Adams-Rieffel theorem we have an isomorphism  $\mathrm{Ext}^n_{KM}(U\otimes_K V,W)\cong \Ext^n_{KM}(U,\Hom_K(V,W))$ for all $n\geq 0$.  In particular, if $U$ is projective, then $\Ext^n_{KM}(U\otimes_K V, W)=0$ for all $n\geq 1$ and $KM$-modules $W$, and so $U\otimes_K V$ is projective.  Similarly, if $W$ is injective, then $\Ext^n_{KM}(U,\Hom_K(V,W))=0$ for all $n\geq 1$ and $KM$-modules $U$, and so $\Hom_K(V,W)$ is injective.
\end{proof}

We shall primarily be interested in the case that $U=K$, which generalizes a well-known result for groups to monoids.

\begin{Cor}\label{c:to.homology}
Let $M$ be a monoid, $V$ a $KG(M)$-module and $W$ a $KM$-module.   Then if $V$ is a projective $K$-module, \[\Ext^n_{KM}(V,W)\cong H^n(M,\Hom_K(V,W)).\]
\end{Cor}
\begin{proof}
It is immediate from the definitions that $K\otimes_K V\cong V$ as a $KM$-module, and so $H^n(M,\Hom_K(V,W))\cong \Ext^n_{KM}(V,W)$ by Theorem~\ref{t:to.triv}.
\end{proof}

We now aim to compute $\Ext$ from a $KG(M)$-module (inflated to $KM$)  to a coinduced module using Theorem~\ref{t:cohomology.coinduced} and Corollary~\ref{c:to.homology}. To do so, we
shall need one last isomorphism.  Let $e\in E(M)$.  The canonical homomorphism $\psi\colon M\to G(M)$ restricts to a group homomorphism $\psi\colon G_e\to G(M)$.  Hence $KG(M)$-modules can be inflated to $KG_e$-modules.

\begin{Prop}\label{p:hom.commute}
Let $M$ be a monoid, $e\in E(M)$ and $K$ a commutative ring. Let $\psi\colon M\to G(M)$ be the canonical homomorphism.  Then, for any $KG(M)$-module $V$ and $KG_e$-module $W$, there is an isomorphism of $KM$-modules $\Hom_K(V,\Coind_e(W))\to \Coind_e(\Hom_K(V,W))$, natural in $V,W$.  Here $V$ is viewed as both a $KM$-module and a $KG_e$-module via inflation along $\psi$.
\end{Prop}
\begin{proof}
Define $F\colon \Hom_K(V,\Coind_e(W))\to \Coind_e(\Hom_K(V,W))$ as follows.  If $\p\colon V\to \Hom_{KG_e}(KR_e,W)$ belongs to $\Hom_K(V,\Coind_e(W))$, then define $F(\p)\colon KR_e\to \Hom_K(V,W)$ by $F(\p)(r)(v) = \p(\psi(r)\inv v)(r)$ for $r\in R_e$.  First we check that $F(\p)$ is a $KG_e$-module homomorphism.  Indeed, we have that $F(\p)(gr)(v) = \p(\psi(gr)\inv v)(gr) = g\p(\psi(r)\inv \psi(g)\inv v)(r) = g(F(\p)(r)(\psi(g)\inv v)) = (g(F(\p)(r)))(v)$.     We next check that $F$ is a $KM$-module homomorphism.  We compute $F(m\p)(r)(v) = ((m\p)(\psi(r)\inv v))(r)$.  But we have $(m\p)(\psi(r)\inv v) = m\p(\psi(m)\inv \psi(r)\inv v)$, and so
\begin{align*}
((m\p)(\psi(r)\inv v))(r) &= \p(\psi(m)\inv \psi(r)\inv v)(rm)\\ &=\begin{cases}\p(\psi(rm)\inv v)(rm), &\text{if}\ rm\in R_e\\ 0, & \text{else.}\end{cases}
\end{align*}
 On  the other hand, $(mF(\p))(r) = F(\p)(rm)$, and this is $0$ if $rm\notin R_e$,
whereas if $rm\in R_e$, then $F(\p)(rm)(v) = \p(\psi(rm)\inv v)(rm)$. Thus $F$ is a $KM$-module homomorphism, as it is obviously $K$-linear.

We define an inverse $G\colon \Coind_e(\Hom_K(V,W))\to \Hom_K(V,\Coind_e(W))$ by $G(\lambda)(v)(r) = \lambda(r)(\psi(r)v)$ for $r\in R_e$, $v\in V$.  Note  that $G(\lambda)(v)\in \Coind_e(W)$, for $v\in V$, because if $g\in G_e$ and $r\in R$, then
\begin{align*}
G(\lambda)(v)(gr) &= \lambda(gr)(\psi(gr)v) = (g\lambda(r))(\psi(g)\psi(r) v)\\ &= g\lambda(r)(\psi(g)\inv \psi(g)\psi(r)v) = g\lambda(r)(\psi(r)v) = gG(\lambda)(v)(r).
\end{align*}
  Let us check that $F,G$ are inverses.  Indeed, if $\p\in \Hom_K(V,\Coind_e(W))$, $r\in R_e$ and $v\in V$, then $GF(\p)(v)(r) = F(\p)(r)(\psi(r)v) = \p(\psi(r)\inv \psi(r)v)(r) = \p(v)(r)$.  Similarly, if $\lambda\in \Coind_e(\Hom_K(V,W))$, then \[FG(\lambda)(r)(v) = G(\lambda)(\psi(r)\inv v)(r) = \lambda(r)(\psi(r)\psi(r)\inv v) = \lambda(r)(v).\]  Thus $F,G$ are inverse isomorphisms.  This completes the proof.
\end{proof}

The next result is the main result of this section.

\begin{Thm}\label{t:ext.coinduced}
Let $M$ be a monoid with group completion $G(M)$,  $e\in E(M)$ and $K$ a commutative ring. Let $V$ be a $KG(M)$-module which is projective as a $K$-module.  Then
\begin{enumerate}
  \item There is a spectral sequence \[\Ext^p_{KG_e}(\til H_{q-1}(\mathcal B(R(e)\times M),K),\Hom_K(V,W))\Rightarrow_p \Ext_{KM}^n(V,\Coind_e(W)).\]
  \item If $M$ is right p.p.\ (e.g., regular), then there is a spectral sequence  \[\Ext^p_{KG_e}(\til H_{q-1}(\Delta(\Omega(R(e))),K),\Hom_K(V,W))\Rightarrow_p  \Ext_{KM}^n(V,\Coind_e(W)).\]
\end{enumerate}
If, in addition,  $W$ is an injective $KG_e$-module, then
\begin{enumerate}[resume]
  \item For all $n\geq 0$, \[\Ext_{KM}^n(V,\Coind_e(W))\cong \Hom_{KG_e}(\til H_{n-1}(\mathcal B(R(e)\times M),K),\Hom_K(V,W)).\]
  \item If $M$ is right p.p.\ (e.g., regular), then \[\Ext_{KM}^n(V,\Coind_e(W))\cong \Hom_{KG_e}(\til H_{n-1}(\Delta(\Omega(R(e))),K),\Hom_K(V,W))\] for all $n\geq 0$.
\end{enumerate}
\end{Thm}
\begin{proof}
Corollary~\ref{c:to.homology} and Proposition~\ref{p:hom.commute} yield $\Ext^n_{KM}(V,\Coind_e(W))\cong H^n(M,\Hom_K(V,\Coind_e(W)))\cong H^n(M,\Coind_e(\Hom_K(V,W)))$.  The result then follows from Theorem~\ref{t:cohomology.coinduced}, noting  that $\Hom_K(V,W)$ is an injective $KG_e$-module whenever $W$ is injective by Theorem~\ref{t:to.triv} applied to $G_e$.
\end{proof}

 If $V$ is a  $KG(M)$-module that is finitely generated and projective over $K$, then $V^*\otimes_K W\cong \Hom_K(V,W)$ as $KM$-modules for any $KM$-module $W$.  Indeed, the standard $K$-linear isomorphism takes $f\otimes w$ to the mapping $v\mapsto f(v)w$.  Therefore, $mf\otimes mw$ is sent to the map taking $v$ to $(mf)(v)mw = f(\psi(m)\inv v)mw = mf(\psi(m)\inv v)w$, which shows this isomorphism is an isomorphism of $KM$-modules.  Moreover, if $M$ is finite and $K$ is a field, then every finite dimensional $KG(M)$-module is finitely generated projective over $K$, and if the characteristic of $K$ does not divide the order of $|G_e|$ for $e\in E(M)$, then every $KG_e$-module is injective. Hence we deduce the following from Theorem~\ref{t:ext.coinduced}.

\begin{Cor}\label{c:main.fd.coind}
Let $M$ be a finite monoid with group completion $G(M)$ and $e\in E(M)$.  Let $K$ be a field,  $V$ a finite dimensional $KG(M)$-module and $W$ a $KG_e$-module.
\begin{enumerate}
  \item There is a spectral sequence \[ \Ext^p_{KG_e}(\til H_{q-1}(\mathcal B(R(e)\rtimes M),K),V^\ast\otimes_K W)\Rightarrow_p \Ext^n_{KM}(V,\Coind_e(W)).\]
  \item If $M$ is right p.p.~(in particular, regular), there is a spectral sequence \[\Ext^p_{KG_e}(\til H_{q-1}(\Delta(\Omega(R(e))),K),V^\ast\otimes_K W)\Rightarrow_p \Ext^n_{KM}(V,\Coind_e(W)).\]
\end{enumerate}
Moreover, if the characteristic of $K$ does not divide $|G_e|$, then
\begin{enumerate}[resume]
  \item For all $n\geq 0$, \[\Ext^n_{KM}(V,\Coind_e(W))\cong \Hom_{KG_e}(\til H_{n-1}(\mathcal B(R(e)\rtimes M),K),V^\ast\otimes_K W).\]
  \item If $M$ is right p.p.~(in particular, regular), then, for all $n\geq 0$, \[\Ext^n_{KM}(V,\Coind_e(W))\cong \Hom_{KG_e}(\til H_{n-1}(\Delta(\Omega(R(e))),K),V^\ast\otimes_K W).\]
\end{enumerate}
\end{Cor}

 Using the duality $D$ from Remark~\ref{r:duality} and Proposition~\ref{p:dual.induced}, we obtain the following from Corollary~\ref{c:main.fd.coind} applied to $M^{op}$.

\begin{Cor}\label{c:ext.ind}
Let $M$ be a finite monoid with group completion $G(M)$ and $e\in E(M)$.  Let $K$ be a field, $V$ a $KG(M)$-module and $W$ a $KG_e$-module, both finite dimensional.
\begin{enumerate}
  \item There is a spectral sequence  \[\Ext^p_{KG_e}(\til H_{q-1}(\mathcal B(L(e)\rtimes M^{op}),K),D(V^\ast)\otimes_K D(W))\Rightarrow_p \Ext^n_{KM}(\Ind_e(W),V)\] where we work in the category of right $KG_e$-modules on the left hand side.
  \item If $M$ is left p.p.~(in particular, regular), there is a spectral sequence \[\Ext^p_{KG_e}(\til H_{q-1}(\Delta(\Omega_{M^{op}}(L(e))),K),D(V^\ast)\otimes_K D(W))\Rightarrow_p \Ext^n_{KM}(\Ind_e(W),V)\] where we work in the category of right $KG_e$-modules on the left hand side.
\end{enumerate}
If the characteristic of $K$ does not divide $|G_e|$, then
\begin{enumerate}[resume]
  \item For all $n\geq 0$, \[\Ext^n_{KM}(\Ind_e(W),V)\cong \Hom_{KG_e}(\til H_{n-1}(\mathcal B(L(e)\rtimes M^{op}),K),D(V^\ast)\otimes_K D(W))\] where $\Hom_{KG_e}$ is taken in the category of right $KG_e$-modules.
  \item If $M$ is left p.p.~(in particular, regular), then, for all $n\geq 0$, \[\Ext^n_{KM}(\Ind_e(W),V)\cong \Hom_{KG_e}(\til H_{n-1}(\Delta(\Omega_{M^{op}}(L(e))),K),D(V^\ast)\otimes_K D(W))\] where $\Hom_{KG_e}$ is taken in the category of right $KG_e$-modules.
\end{enumerate}
\end{Cor}

For a Dedekind-finite monoid $M$ with group of units $G$ and ideal of singular elements $S=M\setminus G$,  we have homomorphisms $\rho\colon KM\to KG$ and $\psi\colon KM\to KG(M)$ where $\rho$ is the identity on $G$ and annihilates $S$ and $\psi$ is induced by the canonical map $M\to G(M)$.  Hence we can inflate both $KG$ and $KG(M)$-modules to $KM$-modules.
Note that since each right invertible element of $M$ is left invertible, we have that $R_1=G_1=G=L_1$ and $R(1)=S=L(1)$.
Moreover, if $W$ is a $KG$-module, then $\Coind_1(W) =\Hom_{KG}(KG,W)\cong W\cong KG\otimes_{KG} W=\Ind_1(W)$ as $KM$-modules.
Therefore, Corollaries~\ref{c:main.fd.coind} and~\ref{c:ext.ind} specialize to the following result, where, for simplicity, we state only the version for fields of good characteristic.

\begin{Cor}\label{c:main.fd}
Let $M$ be a finite monoid with group of units $G$, ideal of singular elements $S$ and group completion $G(M)$.  Let $K$ be a field whose characteristic does not divide $|G|$, $V$ a $KG(M)$-module and $W$ a $KG$-module, both finite dimensional.
\begin{enumerate}
  \item $\Ext^n_{KM}(V,W)\cong \Hom_{KG}(\til H_{n-1}(S\mathcal EM,K),V^\ast\otimes_K W)$ for all $n\geq 0$.
  \item If $M$ is right p.p.~(in particular, regular), then \[\Ext^n_{KM}(V,W)\cong \Hom_{KG}(\til H_{n-1}(\Delta(\Omega(S)),K),V^\ast\otimes_K W).\]
  \item $\Ext^n_{KM}(W,V)\cong \Hom_{KG}(\til H_{n-1}(\mathcal EM^{op}S,K),D(V^\ast)\otimes_K D(W))$ for all $n\geq 0$, where $\Hom_{KG}$ is taken in the category of right $KG$-modules.
  \item If $M$ is left p.p.~(in particular, regular), then \[\Ext^n_{KM}(W,V)\cong \Hom_{KG}(\til H_{n-1}(\Delta(\Omega_{M^{op}}(S)),K),D(V^\ast)\otimes_K D(W))\] where $\Hom_{KG}$ is taken in the category of right $KG$-modules.
\end{enumerate}
\end{Cor}

As a simple example consider the full transformation monoid $T_n$ of all mappings on the set $\{1,\ldots, n\}$ (acting on the left) and put $K=\mathbb C$, the field of complex numbers.  The group of units of $T_n$ is the symmetric group $S_n$.  Then $S=T_n\setminus S_n$ is the ideal of non-surjective mappings.  The group completion $G(T_n)$ is trivial since if $c$ is a constant mapping then $cf=c$ for all $f$, and hence $f$ maps to $1$ in $G(T_n)$.  It is well known and easy to check that $T_n$ is regular and that $fT_n\subseteq gT_n$ if and only if the image of $f$ is contained in the image of $g$~\cite{CP,repbook}.  It follows that $\Omega(S)$ can be identified with the poset of proper nonempty subsets of $\{1,\ldots, n\}$.  Therefore, $\Delta(\Omega(S))$ can be identified with the barycentric subdivision of the simplicial complex with vertex set $\{1,\ldots, n\}$ where any proper nonempty subset is a simplex.  But this is the boundary of the $(n-1)$-simplex with vertex set $\{1,\ldots, n\}$ and all nonempty subsets a simplex. Therefore, $\Delta(\Omega(S))$ is an $(n-2)$-sphere.  Moreover, the $S_n$-action on $\Delta(\Omega(S))$ can be identified with the action  of $S_n$ by simplicial mappings on the boundary of the $(n-1)$-simplex induced by permuting the vertices (after barycentric subdivision).  Since the even permutations preserve orientation and the odd permutations reverse orientation under this action, we deduce that $\til H_q(\Delta(\Omega(S)),\mathbb C)=0$ except when $q=n-2$, in which case it is the sign representation of $S_n$.  Thus, from Corollary~\ref{c:main.fd} we deduce that if $W$ is a simple $\mathbb CS_n$-module, then $\Ext^q_{\mathbb CM}(\mathbb C,W)=0$ unless $q=n-1$ and $W$ is the sign representation, in which case it is $\mathbb C$.
This was already proved by the author in~\cite{globaltn} using the minimal projective resolution of $\mathbb C$ (which is the augmented simplicial chain complex of the $(n-1)$-simplex with vertex set $\{1,\ldots, n\}$, on which $T_n$ acts simplicially).

Corollary~\ref{c:main.fd.coind} will typically be applied when both $V$ and $\Coind_e(W)$ are simple.  Computing $\Ext$ between simple modules for a finite dimensional algebra is an important undertaking.  The computation of $\Ext^1$ is of particular interest because it yields the quiver of the algebra~\cite{assem,benson}.  This is a major impetus for considering spectral sequences in Corollary~\ref{c:main.fd.coind}, as the five term exact sequence associated to this spectral sequence allows us to compute immediately $\Ext^1$.

\begin{Cor}\label{c:ext1}
Let $M$ be a monoid with group completion $G(M)$ and $K$ a commutative ring.  Let $e\in E(M)$, $V$ a $KG(M)$-module finitely generated and projective over $K$, and $W$ a $KG_e$-module.
\begin{enumerate}
  \item  If $eM$ is not a minimal right ideal, then $\Ext^1_{KM}(V,\Coind_e(W))\cong \Hom_{KG_e}(\til H_0(\Delta(\Omega(R(e)),K),V^\ast\otimes_K W)$.
  \item  If $M$ is finite, $K$ is a field, $W$ is finite dimensional and $Me$ is not a minimal left ideal, then \[\Ext^1_{KM}(\Ind_e(W),V)\cong \Hom_{KG_e}(\til H_0(\Delta(\Omega_{M^{op}}(L(e)),K),D(V^\ast)\otimes_K D(W))\] where $\Hom_{KG_e}$ is taken in the category of right $KG_e$-modules.
\end{enumerate}
\end{Cor}
\begin{proof}
For the first item, notice that on the $E_2$-page of the spectral sequence in Theorem~\ref{t:ext.coinduced}(1), we have on the line $q=0$ that \[E^{p,0}_2=\Ext^p_{KG_e}(\til H_{-1}(\mathcal B(R(e)\rtimes M),K),V^\ast\otimes_K W)=0\] as $R(e)\neq \emptyset$ because $eM$ is not minimal.  Therefore, the five term exact sequence of this spectral sequence reduces to an exact sequence \[0 \to \Ext^1_{KM}(V,\Coind_e(W))\to \Hom_{KG_e}(\til H_0(\mathcal B(R(e)\rtimes M),K),V^\ast\otimes_K W)\to 0,\]  or in other words \[\Ext^1_{KM}(V,\Coind_e(W))\cong \Hom_{KG_e}(\til H_0(\mathcal B(R(e)\rtimes M),K),V^\ast\otimes_K W).\]  But $\mathcal B\Phi_{R(e)}\colon \mathcal B(R(e)\rtimes M)\to \Delta(\Omega(R(e)))$ is $G_e$-equivariant and induces a bijection of the sets of path components by Remark~\ref{r:path.comp.bij}, and so \[\til H_0(\mathcal B(R(e)\rtimes M),K)\cong \til H_0(\Delta(\Omega(R(e))),K)\] as $KG_e$-modules.  This completes the proof of the first item.

The second item follows from the first using duality (Remark~\ref{r:duality} and Proposition~\ref{p:dual.induced}).
\end{proof}

The special case of the following result where $M$ is finite and $K$ is a field can be extracted from the proof of~\cite[Lemma~2B.5]{Khovanov} or from~\cite{DO}.  See~\cite[Remark~2B.12]{Khovanov}.  It can also be deduced with some work from Theorem~\ref{t:two.trivials}.

\begin{Cor}\label{c:ext1.two.trivs}
Let $M$ be a Dedekind-finite monoid with group of units $G$ and ideal of singular elements $S=M\setminus G$ with $M\neq G$.   Let $K$ be a commutative ring.   Then $\Ext^1_{KM}(K,K_{(G)})\cong \til H^0(\Delta(\Omega(S)),K)^G$.  If $M$ is finite and $K$ is a field, then $\Ext^1_{KM}(K_{(G)},K)\cong \til H^0(\Delta(\Omega_{M^{op}}(S)),K)^G$.
\end{Cor}
\begin{proof}
Recall that $R(1)=S$, $R_1=G_1=G$ and $\Coind_1(V)=V$ (inflated to $KM$) for a $KG$-module $V$.
Since $K^*$ restricts to $KG$ as the trivial module and the tensor product of two trivial modules is trivial,  Corollary~\ref{c:ext1} yields $\Ext^1_{KM}(K,K_{(G)})\cong \Hom_{KG}(\til H_0(\Delta(\Omega(S)),K),K)$.  Note that, for any $KG$-module $V$, one has $\Hom_{KG}(V,K)\cong \Hom_K(V,K)^G$.  For any space $X$, one has $\til H^0(X,K)\cong \Hom_K(\til H_0(X,K),K)$ as the exact sequence $0\to \til H_0(X,K)\to H_0(X,K)\to K\to 0$ splits, and hence \[0\to K\to \Hom_K(H_0(X,K),K)\to \Hom_K(\til H_0(X,K),K)\to 0\] is exact.  But $\Hom_K(H_0(X,K),K)=\Hom_K(C_0(X,K)/B_0(X,K),K)\cong Z^0(X,K)$, and the image of $K$ consists of the constant mappings.  Thus $\til H^0(X,K)\cong \Hom_K(\til H_0(X,K),K)$. Moreover, this isomorphism is natural in $X$.  Therefore, we have $\Hom_K(\til H_0(\Delta(\Omega(S)),K),K)\cong \til H^0(\Delta(\Omega(S)),K)$ as right $KG$-modules, and so putting it all together  yields $\Ext^1_{KM}(K,K_{(G)})\cong \Hom_{KG}(\til H_0(\Delta(\Omega(S)),K),K)\cong \til H^0(\Delta(\Omega(S)),K)^G$.

The second item follows from the first and Remark~\ref{r:duality}.
\end{proof}

\section{The global dimension of the affine monoid}\label{s:affine}
Recall that the (left) \emph{global dimension} of a ring $R$ is the supremum of the projective (equivalently, injective) dimensions of (left) $R$-modules. Alternatively, it is the supremum of the integers $n$ such that $\Ext^n_R(M,N)\neq 0$ for some $R$-modules $M,N$.    If $A$ is a finite dimensional algebra over a field, then it is well known that the global dimension of $A$ is the maximum projective (equivalently, injective) dimension of a simple $A$-module (which might be infinite) or, equivalently, the supremum of the integers $n$ such that $\Ext^n_A(S,S')\neq 0$ for some simple modules $S,S'$; see~\cite{assem} or~\cite[Corollary~16.2]{repbook}.  Note that if $G$ is a group and $K$ is a field, then the global dimension of $KG$ is the $K$-cohomological dimension of $G$ (that is, the projective dimension of the trivial module $K$) as a consequence of Corollary~\ref{c:to.homology} applied to $M=G$.  Unfortunately, this is not in general the case for monoids.  The trivial $KM$-module is projective whenever $M$ has a right zero element but there are many such examples for which the global dimension is not $0$ (e.g., if $M=\{1,x,0\}$ with $x^2=0$, then $KM$ has infinite global dimension over any field $K$, but $M$ has $K$-cohomological dimension zero).

Let $q$ be a prime power and $\mathrm{Aff}(n,q)$ be the monoid of all affine transformations of $\mathbb F_q^n$.   Elements of $\mathrm{Aff}(n,q)$ are of the form $x\mapsto Ax+b$ where $A\in M_n(\mathbb F_q)$ and $b\in \mathbb F_q^n$.  In fact, $\mathrm{Aff}(n,q)$ is isomorphic to the semidirect product $\mathbb F_q^n\rtimes M_n(\mathbb F_q)$.  The group of units of $\mathrm{Aff}(n,q)$ is the affine general linear group $\mathrm{AGL}(n,q)\cong \mathbb F_q^n\rtimes \mathrm{GL}_n(\mathbb F_q)$. We shall prove in this section, using our topological methods and some nontrivial results from monoid representation theory, that the global dimension of $K\mathrm{Aff}(n,q)$ is $n$ whenever $K$ is a field whose characteristic does not divide $|\mathrm{AGL}(n,q)|$.  We remark that when $n=1$, this follows already from results of Margolis and the author in~\cite{rrbg}.  Note that $K\mathrm{Aff}(n,q)$ has infinite global dimension if the characteristic of $K$ divides $|\mathrm{AGL}(n,q)|$ by a result of Nico~\cite{Nico1}; see~\cite[Theorem~16.8]{repbook} for details.

We begin with some well-known facts about the semidirect product of a monoid with a group where the monoid is acting on the group.

\begin{Prop}\label{p:basic.semidirect}
Let $M$ be a monoid acting by endomorphisms on a group $G$. We write $m(g)$ for the action of $m\in M$ on $g\in G$.  Let $N=G\rtimes M$ and let $\pi\colon N\to M$ be the projection.
\begin{enumerate}
  \item If $M$ is regular, then $N$ is regular.
  \item $N(g,m)\subseteq N(h,m')$ if and only if $Mm\subseteq Mm'$.
  \item $(g,m)\in N$ is idempotent if and only if $m\in E(M)$ and $m(g)=1$.
  \item If $e\in E(M)$ is an idempotent, then the maximal subgroup at $(1,e)$ in $N$ is $e(G)\rtimes G_e$.
  \item $(g,m)N\subseteq (h,m')N$ if and only if $mM\subseteq m'M$ and $gm(G)\subseteq hm'(G)$.
  \item $N(g,m)N \subseteq N(h,m')N$ if and only if $MmM\subseteq Mm'M$.
\end{enumerate}
\end{Prop}
\begin{proof}
Suppose that $(g,m)\in N$ and $mam=m$ with $a\in M$.  Then
\[
(g,m)(a(g\inv),a)(g,m) = (gma(g\inv)ma(g),mam)=(g,m),\]
 and so $N$ is regular.

Clearly $N(g,m)\subseteq N(h,m')$ implies $Mm\subseteq Mm'$ since $\pi$ is a homomorphism.  Conversely, if $m=am'$, then \[(ga(h\inv),a)(h,m') = (ga(h\inv)a(h),am') = (g,m),\] and so $(g,m)\in N(h,m')$.

We compute $(g,m)^2=(gm(g),m^2)$ and so $(g,m)=(g,m)^2$ if and only if $m=m^2$ and $m(g)=1$.

We claim that if $e\in E(M)$, then $(1,e)N(1,e)= e(G)\rtimes eMe$. First of all the semidirect product $e(G)\rtimes eMe$ is a monoid with identity $(1,e)$ because $e$ acts as the identity on $e(G)$. Clearly,  $e(G)\rtimes eMe$ is a submonoid of $(1,e)N(1,e)$.  But $(1,e)(g,m)(1,e) = (e(g),eme)$ and so $(1,e)N(1,e)=e(G)\rtimes eMe$.  It follows that the group of units of $(1,e)N(1,e)$ is $e(G)\rtimes G_e$ as the semidirect product of two groups is a group.

If $(g,m)=(h,m')(x,a)$, then $m=m'a$ and $g=hm'(x)$.  Therefore,  we have $mM\subseteq m'M$,  $m(G)=m'(a(G))\subseteq m'(G)$ and $g\in hm'(G)$.  But then $gm(G)\subseteq hm'(G)$.  Conversely, if $mM\subseteq m'M$ and $gm(G)\subseteq hm'(G)$, then we can write $m=m'a$ and $g=hm'(x)$ with $a\in M$ and $x\in G$.  Then $(h,m')(x,a) = (hm'(x),m'a) = (g,m)$.

Since $\pi$ is a homomorphism, $N(g,m)N\subseteq N(h,m')N$ implies that $MmM\subseteq Mm'M$.  For the converse, suppose that $m=xm'y$. Then \[(g,m)=(gx(h\inv),x)(h,m')(1,y).\]  Thus $N(g,m)N\subseteq N(h,m')N$.
\end{proof}

The following is an immediate consequence of Theorem~\ref{t:semidirect}.

\begin{Cor}\label{c:homological.epi}
Let $M$ be a finite monoid acting on a finite group $G$ and let $K$ be a field whose characteristic does not divide $|G|$.  Then $\pi\colon G\rtimes M\to M$ induces a homological epimorphism $\ov \pi\colon K[G\rtimes M]\to KM$ and hence, for any pair of $KM$-modules $V,W$, we have $\Ext^n_{K[G\rtimes M]}(V,W)\cong \Ext^n_{KM}(V,W)$.
\end{Cor}
\begin{proof}
Since $KG$ is semisimple, $H_n(G,K)\cong \Tor^{KG}_n(K,K)=0$ for $n\geq 1$.  The result follows from Theorem~\ref{t:semidirect}.
\end{proof}

In what follows we shall use some structural finite semigroup theory, all of which can be found in~\cite[Appendix~A]{qtheor},~\cite{Arbib} or~\cite[Chapter~1]{repbook}. Elements $a,b\in M$ are $\mathscr J$-equivalent if $MaM=MbM$, $\mathscr L$-equivalent if $Ma=Mb$ and $\mathscr R$-equivalent if $aM=bM$.  A $\mathscr J$-class $J$ is called \emph{regular} if it contains an idempotent or, equivalently, each element of $J$ is regular.  The $\mathscr J$-class of an element $m\in M$ is denoted by $J_m$.  Any $\mathscr L$- or $\mathscr R$-class of a regular $\J$-class contains an idempotent.  The set of $\mathscr J$-classes is partially ordered by $J\leq J'$ if $MJM\subseteq MJ'M$.
Let $M$ be a finite monoid and $J$ a regular $\mathscr J$-class of $M$.  Fix an idempotent $e\in J$ and let $G_e$ be the maximal subgroup at $e$.  Note that $G_e=eMe\cap J$.  Also, one has that $m\J n$ if and only there is $x\in M$ with $m\eL x\R n$, if and only if there is $y\in M$ with $m\R y\eL n$.
Consequently, each $\mathscr R$-class of $J$ has nonempty intersection with each $\mathscr L$-class of $J$.

 Let $A$ be an index set for the set of $\mathscr R$-classes in $J$ and $B$ be an index set for the set of $\mathscr L$-classes in $J$.  For convenience we assume that the index for both the $\R$-class and $\eL$-class of $e$ is $1$.  Green's relation $\mathscr H$~\cite{Green} is the intersection ${\eL}\cap {\R}$.  We denote by $H_{ab}$ the $\mathscr H$-class $R_a\cap L_b$.  Note that $H_{11}=G_e$.  If we choose a representative $r_a\eL e$ for the $\mathscr R$-class $R_a$ with $a\in A$ and a representative $\ell_b\R e$ for each $\mathscr L$-class $L_b$ with $b\in B$, then each element of $m\in J$ can be written uniquely in the form $m=r_ag\ell_b$ where $m\in H_{ab}$ and $g\in G_e$.  For $b\in B$ and $a\in A$, we have that $\ell_b\in eM$ and $r_a\in Me$.  Therefore, $\ell_br_a\in eMe$.  Moreover, since $J\cap eMe=G_e$, we have that $\ell_br_a\in J$ if and only if $\ell_br_a\in G_e$.    Define a  $B\times A$-matrix $P$ with entries in $G_e\cup \{0\}$ (where $0$ is a symbol not in $G_e$ that will shortly be identified with the zero of the group algebra of $G_e$ over some base commutative ring) by
 \[P_{ba} = \begin{cases}\ell_br_a, & \text{if}\ \ell_br_a\in G_e\\ 0, & \text{else.}\end{cases}\]
 The matrix $P$ is called the \emph{sandwich matrix} of the $\mathscr J$-class $J$. Note that if $m=r_ag\ell_b$ and $n=r_{a'}g'\ell_{b'}$ are in $J$, then $mn\in J$ if and only if $P_{ba'}\neq 0$, in which case $mn=r_agP_{ba'}g'\ell_{b'}$.

 If $K$ is a commutative ring, then we can view $P$ as a matrix over $KG_e$.  In particular, we can talk about right and left invertibility of $P$. Note that right and left invertibility of the sandwich matrix depends only on the $\J$-class and not the choices we have made.  Sandwich matrices have played an important role in the representation theory of finite monoids since the beginning of the subject~\cite{Clifford2,Munn1,Poni}.  For example, the algebra of a finite monoid $M$ over a field $K$ is semisimple if and only if $M$ is regular, the characteristic of the field divides the order of no maximal subgroup of $M$ and each sandwich matrix is invertible over the corresponding group algebra; see~\cite[Theorem~5.21]{repbook}.

The \emph{length} of a finite chain $C$ in a poset is $|C|-1$, i.e., the dimension of the corresponding simplex in the order complex.
 The following theorem is~\cite[Theorem~4.4]{rrbg}, dualized since we are working with left modules instead of right modules. The proof of this theorem in~\cite{rrbg} used Nico's theorem~\cite{Nico2}, but unfortunately with an incorrect formulation.  While the desired result can, indeed, be deduced from Nico's work~\cite{Nico2}, a more careful argument is needed.  For the sake of completeness, we instead give a proof here deducing this from $KM$ being a directed quasi-hereditary algebra.  Let $\pd V$ denote the projective dimension of an $A$-module $V$.  Note that
 \[\pd V=\sup\{n\mid \Ext^n_A(V,-)\neq 0\}\] by a standard argument.

 \begin{Thm}\label{t:directed}
 Let $M$ be a finite regular monoid and $K$ a field whose characteristic divides the order of no maximal subgroup of $M$.  If each sandwich matrix of $M$ is right invertible over the group algebra of the corresponding maximal subgroup over $K$, then the global dimension of $KM$ is bounded by $n$ where $n$ is the length of the longest chain of principal two-sided ideals of $M$.
 \end{Thm}
 \begin{proof}
Recall that the \emph{apex}~\cite[Chapter~5]{repbook} of a simple $KM$-module $V$ is the unique minimal $\mathscr J$-class $J$ of $M$ with $JV\neq 0$.
 Define the \emph{height} $\mathrm{ht}(J)$ of a $\mathscr J$-class $J$ of $M$ to be the length of the longest chain of $\mathscr J$-classes of $M$ whose smallest element is $J$.  We prove that if $S$ is a simple $KM$-module with apex $J$, then $\pd S\leq \mathrm{ht}(J)$ by induction on $\mathrm{ht}(J)$.  Under the hypothesis we are assuming, it is shown in~\cite[Theorem~4.4]{rrbg} that if $e$ is an idempotent and $S$ is a simple $KM$-module with apex $e$, then the natural map $\Ind_e(eS)\to S$ is the projective cover.

 Note that $\mathrm{ht}(J_e)=0$ if and only if $e=1$.  The simple modules with apex $J_1$ are inflations of a simple $KG_1$-module $V$ to $M$.  But then, since $L_1=G_1$, we have that $\Ind_1(V)=KG_1\otimes_{KG_1}V\cong V$.  Thus $V$ is a projective $KM$-module, and so $\pd V=0=\mathrm{ht}(J_1)$.  Suppose that $\pd S'\leq \mathrm{ht}(J)$ whenever $S'$ is a simple $KM$-module with apex $J$ of height at most $r$ and that $S$ is a simple $KM$-module with apex $J_e$ of height $r+1$  with $e\in E(M)$.  Then we have an exact sequence $0\to \rad(\Ind_e(eS))\to \Ind_e(eS)\to S\to 0$ with $\Ind_e(eS)$ projective.   Thus $\pd S\leq \pd \rad(\Ind_e(eS))+1$.  But by~\cite[Theorem~5.5]{repbook} every composition factor of $\rad (\Ind_e(eS))$ has apex strictly larger than $J_e$.  By induction and~\cite[Lemma~16.1(ii)]{repbook}, it follows that $\Ext^{r+1}_{KM}(\rad(\Ind_e(eS)),-)=0$ and hence $\pd(\rad\Ind_e(eS))\leq r$.  Therefore, $\pd S\leq r+1$, as required.  This completes the proof.
 \end{proof}

 We now show that the hypothesis of Theorem~\ref{t:directed} is preserved under taking a semidirect product with a group.

 \begin{Prop}\label{p:semidirect.directed}
 Let $M$ be a finite regular monoid, $G$ a group and $K$ a field whose characteristic does not divide $|G|$ or the order of any maximal subgroup of $M$.  Suppose that $M$ acts on $G$ by endomorphisms and let $N=G\rtimes M$ be the semidirect product.  Suppose that each sandwich matrix of $M$ is right invertible over the group algebra of the corresponding maximal subgroup over $K$. Then the same is true for $N$ and the global dimension of $KN$ is at most $n$ where $n$ is the length of the longest chain of  principal two-sided ideals of $M$.
 \end{Prop}
 \begin{proof}
 Let $\pi\colon N\to M$ be the projection and let us identify $M$ with the submonoid $\{1\}\times M$ of $N$. Proposition~\ref{p:basic.semidirect}(1) shows that $N$ is regular.  Proposition~\ref{p:basic.semidirect}(2) and~(6) imply that each $\eL$-class (respectively, $\J$-class) of $N$ intersects $M$ in an $\eL$-class (respectively, $\J$-class) of $M$ and that $NmN\subseteq Nm'N$ if and only $MmM\subseteq Mm'M$ for $m,m'\in M$.  Hence $n$ is also the length of the longest chain of principal two-sided ideals of $N$. Note that $\R$-classes of $N$ do not have to intersect $M$.

Each $\J$-class $J$ of $N$ contains an idempotent $e$ from $M$ and the maximal subgroup of $N$ at $e$ is $e(G)\rtimes G_e$ where $G_e$ is the maximal subgroup of $M$ at $e$ by Proposition~\ref{p:basic.semidirect}(4).  It also follows from Proposition~\ref{p:basic.semidirect}  (or standard semigroup theory since $M,N$ are regular; cf.~\cite[Exercise~1.10]{repbook}) that $mN=m'N$ (respectively, $Nm=Nm'$) if and only if $mM=m'M$ (respectively, $Mm=Mm'$) for $m,m'\in M$.  Therefore, if an $\R$-class of $N$ intersects $M$, then that intersection is an $\R$-class of $M$.    Let $A$ be an indexing set for the $\R$-classes of $J$ and $B$ an indexing set for the $\eL$-classes.  Let $A'$ be the subset of indices of $\mathscr R$-classes intersecting $M$ and let $A''=A\setminus A'$.   Then we can choose our $\eL$-class representatives $\ell_b\in R_e$, for $b\in B$, to belong to $M$ (since each $\eL$-class intersects $M$) and we can choose our $\R$-class representatives $r_a\in L_e$ with $a\in A'$ to belong to $M$.  By the above discussion, we have that the $\ell_b$, $b\in B$, and the $r_a$, $a\in A'$, form a set of $\eL$-class and $\R$-class representatives of the $\J$-class $J\cap M$ of $M$.  From the construction of the sandwich matrix, we deduce that the sandwich matrix of the $\J$-class $J$ of $N$ has the block form \[P=\begin{bmatrix} P' & P''\end{bmatrix}\] where $P'$ is the $B\times A'$ sandwich matrix of the $\J$-class $J\cap M$ of $M$.  Note that $P'$ has entries in $KG_e$, whereas $P''$ is a $B\times A''$-matrix with entries in $K[e(G)\rtimes G_e]$.  By assumption, there is an $A'\times B$ matrix $Q'$ over  $KG_e$ with $PQ'=I$.  Therefore, the $A\times B$-matrix
\[Q= \begin{bmatrix} Q'\\ 0\end{bmatrix}\] over $KG_e\subseteq K[e(G)\rtimes G_e]$ (where $0$ is the $A''\times B$ zero matrix) is a right inverse to $P$.  The result now follows from Theorem~\ref{t:directed} since $|e(G)\rtimes G_e|$ divides $|G|\cdot |G_e|$.
 \end{proof}

We now apply our theory to compute the global dimension of $K\mathrm{Aff}(n,q)$ over nice fields.

\begin{Thm}\label{t:globdim}
Let $K$ be a field whose characteristic does not divide the order of $\mathrm{AGL}(n,q)$.  Then $K\mathrm{Aff}(n,q)$ has global dimension $n$ and, in fact, the trivial module has projective dimension $n$.
\end{Thm}
\begin{proof}
We shall identify  $\mathrm{Aff}(n,q)$ with $\mathbb F_q^n\rtimes M_n(\mathbb F_q)$.  Note that $M_n(\mathbb F_q)$ is regular being a semisimple ring.
By a theorem of Kov\'acs~\cite{Kovacs} (generalizing a characteristic $0$ result of Okni\'nski and Putcha~\cite{putchasemisimple}) $KM_n(\mathbb F_q)$ is semisimple (see also~\cite[Chapter~5]{repbook} or, for an alternate proof,~\cite{determinantsgp}).  The $\mathscr J$-classes of $M_n(\mathbb F_q)$ are the $J_r$ with $0\leq r\leq n$, where $J_r$ consists of the rank $r$ matrices, and the corresponding principal ideals form a chain.  The maximal subgroup of $J_r$ is isomorphic to $\mathrm{GL}_r(\mathbb F_q)$ (take as an idempotent in $J_r$ the diagonal matrix $e_r$ with $r$ ones followed by $n-r$ zeroes).  Each sandwich matrix is invertible over the maximal subgroup due to the semisimplicity of $KM_n(\mathbb F_q)$, cf.~\cite[Theorem~5.21]{repbook}.  We deduce from Proposition~\ref{p:basic.semidirect} that $\mathrm{Aff}(n,q)$ is regular and the maximal subgroups are of the form $\mathrm{AGL}(r,q)$ with $0\leq r\leq n$ (since $e_r\mathbb F_q^n\cong \mathbb F_q^r$), and hence do not have order divisible by the characteristic of $K$.    From Proposition~\ref{p:semidirect.directed} we conclude that the global dimension of $K\mathrm{Aff}(n,q)$ is at most $n$.

We now show that the trivial module $K$ has projective dimension $n$.  Put $M=\mathrm{Aff}(n,q)$, $G=\mathrm{AGL}(n,q)$ and $S=M\setminus G$.  From our global dimension upper bound, we know that $K$ has projective dimension at most $n$.  So it suffices to show that $\Ext^n_{KM}(K,V)\neq 0$ for some $KM$-module $V$.  If $T(x)=Ax+b$ and $T'(x)=A'x+b'$ with $A,A'\in M_n(\mathbb F_q)$, then $TM\subseteq T'M$ if and only if $AM_n(\mathbb F_q)\subseteq A'M_n(\mathbb F_q)$ and $T(\mathbb F_q^n) \subseteq  T'(\mathbb F_q^n)$ by Proposition~\ref{p:basic.semidirect}(5).  But it is well known and easy to check that $AM_n(\mathbb F_q)\subseteq A'M_n(\mathbb F_q)$ if and only if the image of $A$ is contained in the image of $A'$.  It thus follows that $TM\subseteq T'M$ if and only if $T(\mathbb F_q^n)\subseteq T(\mathbb F_q^n)$ since $A\mathbb F_q^n+b\subseteq A'\mathbb F_q^n+b'$ implies that $A\mathbb F_q^n\subseteq A'\mathbb F_q^n$.  One can also deduce this from~\cite{indep.algebra}, which describes Green's relations on the endomorphism monoid of an independence algebra (of which this is an example). The images of affine transformations are precisely the nonempty affine subspaces of $\mathbb F_q^n$. We conclude that $\Omega(S)$ is isomorphic to the poset of nonempty proper affine subspaces of $\mathbb F_q^n$, which is the  proper part of the geometric lattice of affine subspaces of $\mathbb F_q^n$ ordered by inclusion (this is the lattice of flats of the matroid with point set $\mathbb F_q^n$ and with independent sets the affinely independent sets).   The well-known Folkman theorem says that the order complex of the proper part of a geometric lattice is homotopy equivalent to a wedge of spheres of dimension the length of the longest chain in the proper part (i.e., the dimension of the order complex); in fact, the order complex is shellable, cf.~\cite{Bjornershellable}.  In our setting, the longest chain has length $n-1$, and so $\Delta(\Omega(S))$ has the homotopy type of a wedge of $(n-1)$-spheres.  The number of spheres can be determined from the M\"obius function of the geometric lattice, and, for this particular geometric lattice, the order complex is known to be homotopy equivalent to a wedge of $\prod_{i=1}^n(q^i-1)$-many $(n-1)$-spheres (see~\cite{SolomonAffine},~\cite{Moller} and in the case $q$ is a prime~\cite[Section~8.2]{Browncoset}).  Let $V$ be the $KG$-module $\til H_{n-1}(\Delta(\Omega(S)),K)$.   Then $\dim V = \prod_{i=1}^n(q^i-1)\neq 0$, and so $V$ is a nonzero $KG$-module, which we can then inflate to $M$.  We then have $\Ext^n_{KM}(K,V)\cong \Hom_{KG}(\til H_{n-1}(\Delta(\Omega(S)),K),V)=\End_{KG}(V)\neq 0$ by Corollary~\ref{c:main.fd}.  This completes the proof.
\end{proof}

\begin{Rmk}\label{r:solomon}
In fact, $\til H_{n-1}(\Delta(\Omega(S)),\mathbb C)$ is an irreducible representation of $\mathbb C\mathrm{AGL}(n,q)$, as was first observed by Solomon~\cite{SolomonAffine} (a complete proof can be extracted from~\cite{SolomonAffineBruhat,Siegel}).  The corresponding simple $\mathbb C\mathrm{Aff}(n,q)$-module then has injective dimension $n$.
\end{Rmk}

Let us now compute $\Ext$ from the trivial $\mathbb C\mathrm{Aff}(n,q)$-module to an arbitrary simple module.   The first part of the following result is well known and can be extracted from~\cite{RhodesZalc} or~\cite[Chapter~5]{repbook};  it is also observed in~\cite[Section~4]{rrbg} (but with the conventions chosen for right modules, so we must dualize).  The second part is a combination of~\cite[Theorem~3.7]{rrbg} and~\cite[Corollary~4.5]{rrbg}, dualized to handle left modules.

\begin{Prop}\label{p:MS2}
Let $M$ be a finite regular monoid and $K$ a field whose characteristic does not divide the order of any maximal subgroup of $M$.  Fix for each $\mathscr J$-class an idempotent $e_J$.  Then if the sandwich matrix of each regular $\mathscr J$-class $J$ of $M$ is right invertible over $KG_{e_J}$,  the simple $KM$-modules are (up to isomorphism) the modules $\Coind_{e_J}(V)$ with $V$ a simple $KG_{e_J}$-module for some $\mathscr J$-class $J$. Moreover, under these hypotheses,  if $V, V'$ are simple $KM$-modules with apexes $J,J'$, respectively, and it is not the case that $J<J'$, then $\Ext^n_{KM}(V,W)=0$.
\end{Prop}

Denote by $J_r$ the $\mathscr J$-class of rank $r$ transformations in $\mathrm{Aff}(n,q)$ (i.e., transformations with image an affine subspace of dimension $r$).  We can take as our idempotent representative of $J_r$ the idempotent matrix $e_r$ that fixes the first $r$ standard basis vectors and annihilates the remaining $n-r$ standard basis vectors.  Then $G_{e_r}$ can be naturally identified with $\mathrm{AGL}(r,q)$.  Since the sandwich matrix of $J_r$ is right invertible over $\mathbb C\mathrm{AGL}(r,q)$ by Proposition~\ref{p:semidirect.directed} and the semisimplicity of $\mathbb CM_n(\mathbb F_q)$~\cite{putchasemisimple,Kovacs,repbook}, we deduce from Proposition~\ref{p:MS2} that the simple $\mathbb C\mathrm{Aff}(n,q)$-modules are the modules $\Coind_{e_r}(V)$ with $V$ a simple $\mathbb C\mathrm{AGL}(r,q)$-module for $0\leq r\leq n$.  We now compute the cohomology of any such simple module.

\begin{Cor}
Let $q$ be a prime power and $n\geq 1$.  If $0\leq r\leq n$ and $V$ is a simple $\mathbb C\mathrm{AGL}(r,q)$-module, then $\mathrm{Ext}^m_{\mathbb C\mathrm{Aff}(n,q)}(\mathbb C,\Coind_{e_r}(V))=0$ unless $m=r$ and $V$ is isomorphic to the $(r-1)$-reduced homology module for the action of $\mathrm{AGL}(r,q)$ on the order complex of the poset of proper affine subspaces of $\mathbb F_q^r$, in which case this $\Ext$-space is $\mathbb C$.
\end{Cor}
\begin{proof}
Note that the poset $\Omega(e_r\mathrm{Aff}(n,q))$ can be identified with the poset  of nonempty affine subspaces of $\mathbb F_q^r$ (viewed as $e_r\mathbb F_q^n$) by the discussion in the proof of Theorem~\ref{t:globdim}, and the $G_{e_r}$-action is the usual action of $\mathrm{AGL}(r,q)$ on this poset.  The poset $\Omega(R(e_r))$ can then be identified with the poset of proper nonempty affine subspaces of $\mathbb F_q^r$.  Theorem~\ref{t:cohomology.coinduced} shows that
\begin{align*}
\Ext_{\mathbb C\mathrm{Aff}(n,q)}^m(\mathbb C,\Coind_{e_r}(V))&=H^m(\mathrm{Aff}(n,q),\Coind_{e_r}(V))\\ &\cong \Hom_{\mathbb CG_{e_r}}(\til H_{m-1}(\Delta(R(e_r)),\mathbb C),V)
\end{align*}
 and the result follows from Solomon's theorem~\cite{SolomonAffine}, discussed in Remark~\ref{r:solomon}, and Schur's lemma (as $\Delta(\Omega(R(e_r)))$ is homotopy equivalent to a wedge of $(r-1)$-spheres).
\end{proof}

The proof of Theorem~\ref{t:globdim}, together with Proposition~\ref{p:MS2}, yields the following vanishing result for $\Ext$.

\begin{Cor}
Let $q$ be a prime power and $n\geq 1$.  If $K$ is a field whose characteristic does not divide $|\mathrm{AGL}(n,q)|$, then $\Ext^r_{K\mathrm{Aff}(n,q)}(V,V')=0$ whenever $V,V'$ are simple $K\mathrm{Aff}(n,q)$-modules with respective apexes $J,J'$ that do not satisfy $J<J'$.
\end{Cor}

Finally, using Corollary~\ref{c:homological.epi}, we can prove one last result about $\mathrm{Ext}$ between simple $K\mathrm{Aff}(n,q)$-modules.

\begin{Cor}\label{c:ext.inflate.mn}
Let $q=p^m$ with $p$ prime, $m\geq 1$, and let $K$ be a field of characteristic different than $p$.   Then for any pair $V,W$ of $KM_n(\mathbb F_q)$-modules, inflated to $\mathrm{Aff}(n,q)$-modules via the projection $\mathrm{Aff}(n,q)\to M_n(\mathbb F_q)$, we have that $\Ext^n_{K\mathrm{Aff}(n,q)}(V,W)\cong \Ext^n_{KM_n(\mathbb F_q)}(V,W)$.  In particular, if the characteristic of $K$ does not divide $|\mathrm{GL}_n(\mathbb F_q)|$, then $\Ext^n_{K\mathrm{Aff}(n,q)}(V,W)=0$ for all $n\geq 1$.
\end{Cor}
\begin{proof}
The first statement follows from Corollary~\ref{c:homological.epi}.  The final statement follows because $KM_n(\mathbb F_q)$ is semisimple under these hypotheses by Kov\'acs's theorem~\cite{Kovacs}.
\end{proof}

\section{Projective resolutions from topology}
In this section, we generalize results from~\cite{ourmemoirs} to construct, via topology, an explicit projective resolution of the trivial module over fields of good characteristic for algebras of regular monoids for which each sandwich matrix is right invertible. More generally, we give a resolution by modules with a standard filtration for the natural quasi-hereditary structure on the monoid algebra.  This allows us to give an alternate proof of Corollary~\ref{c:main.fd.coind}(4) for finite regular monoids using resolutions coming from topology that are acyclic with respect to coinduced modules.

Let us say that a $KM$-module is an \emph{induced module} if it is isomorphic to one of the form $\Ind_e(V)$ with $V$ a $KG_e$-module.
First we need a lemma to recognize induced modules. Recall that if $e\in E(M)$, then $L(e)=Me\setminus L_e$ is a left ideal.  If $A\subseteq M$ is a subset and $V$ is a $KM$-module, then we denote by $AV$ the $K$-span of the elements $av$ with $a\in A$ and $v\in V$.

\begin{Lemma}\label{l:induced}
Let $M$ be a monoid and $K$ a commutative ring with unit.  Let $V$ be a $KM$-module and $e\in E(M)$.  Then $V\cong \Ind_e(eV)$ if and only if:
\begin{enumerate}
\item $MeV=V$;
\item $L(e)V=0$;
\item $tG_eeV\cap \sum_{tG_e\neq t'G_e\in L_e/G_e} t'G_eeV=0$ for all $t\in L_e$.
\end{enumerate}
\end{Lemma}
\begin{proof}
Fix a set $T$ of orbit representatives of $G_e$ on $L_e$ with $e\in T$. Recall that $L_e$ is a free right $G_e$-set.  Therefore, if $W$ is a $KG_e$-module, then as a $K$-module, $\Ind_e(W)=\bigoplus_{t\in T} t\otimes W$.  To prove the ``only if'' statement, note that $e(\Ind_e(W))=e\otimes W\cong W$ as $KG_e$-modules by~\cite[Propostion 4.5]{repbook}, and so $\Ind_e(W)\cong \Ind_e(e(\Ind_e(W)))$.  By definition $L(e)\Ind_e(W)=0$, and since $tG_e(e\Ind_e(W))=t\otimes W$, also (1) and (3) are clear.

For the converse, assume that $V$ satisfies (1)--(3).  Define a $K$-bilinear map $\psi\colon KMe\times eV\to V$ by $\psi(a,v)=av$ for $a\in KMe$.     Since $L(e)V=0$, it follows that if $a\in KL(e)$, then $\psi(a,v)=0$ for all $v\in eV$.  Thus $\psi$ induces a well-defined $K$-bilinear map $\beta\colon KL_e\times eV\to V$ (recall that $KL_e=KMe/KL(e)$) with $\beta(a,v)=av$ for $a\in KL_e$.  Clearly, if $g\in G_e$, then $\beta(ag,v)=agv=\beta(a,gv)$, and so $\beta$ induces a homomorphism $\gamma\colon \Ind_e(eV)\to V$ with $\gamma(a\otimes v)=av$ for $a\in KL_e$ and $v\in V$.  Note that $\gamma$ is $KM$-linear as $\gamma(m(a\otimes v))=\psi(ma,v) = mav=m\psi(a,v)=m\gamma(a\otimes v)$ for $m\in M$.

We show that $\gamma$ is onto.  If $v\in V$, then by (1), we can write $v=\sum_{i=1}^n m_iev_i$ with $v_i\in V$ and $m_i\in M$.  If $m_ie\in L(e)$, then $m_iev_i=0$, and so we may assume without loss of generality that each $m_ie\in L_e$.  Then $v=\gamma(\sum_{i=1}^n m_ie\otimes ev_i)$.  To see that $\gamma$ is injective, we use that $\Ind_e(eV)=\bigoplus_{t\in T} t\otimes eV$ as a $K$-module.   Clearly, $\gamma(t\otimes eV) = teV=tG_eeV$.  It follows from (3) that if $x=\sum_{t\in T}t\otimes v_t$ belongs to  $\ker \gamma$ with each $v_t\in eV$, then $tv_t=0$ for all $t\in T$.  Choose $m_t\in M$ with $m_tt=e$ (as $t\in L_e$).  Then $0=m_ttv_t=ev_t=v_t$.  Thus $\gamma$ is injective.  This completes the proof.
\end{proof}

We mostly will use the following corollary.

\begin{Cor}\label{c:action.is.induced}
Let $M$ be a monoid, $K$ a commutative ring and $(X,Y)$ an $M$-set pair.  Suppose that $e\in E(M)$ is such that:
\begin{enumerate}
  \item $X\setminus Y\subseteq MeX$;
  \item $L(e)X\subseteq Y$;
  \item If $x,x'\in eX\setminus eY$ and $m,m'\in L_e$ with $mx=m'x'$, then $mG_e=m'G_e$.
\end{enumerate}
Then $KX/KY\cong \Ind_e(K[eX\setminus eY])$.
\end{Cor}
\begin{proof}
Note that $eX\setminus eY$ is a $G_e$-set and that $e(KX/KY)\cong K[eX\setminus eY]$ as $KG_e$-modules.   We verify that $KX/KY$ satisfies the three properties in Lemma~\ref{l:induced}.  First of all, the cosets of elements of $X\setminus Y$ form a $K$-basis for $KX/KY$, and so it is immediate from (1) that $Me(KX/KY)=KX/KY$.  Trivially, (2) implies $L(e)(KX/KY)=0$.  Finally, note that if $x\in eX\setminus eY$ and $m\in L_e$, then $x=ex\in Mmx$, and so $mx\notin Y$.  This and (3) imply that the sets $tG_e(eX\setminus eY)$, with $tG_e\in L_e/G_e$,  are pairwise disjoint nonempty subsets of $X\setminus Y$.  Recalling that the cosets of $X\setminus Y$ form a $K$-basis for $KX/KY$, we deduce that $\sum_{tG_e\in L_e/G_e} tG_ee(KX/KY)=\bigoplus_{tG_e\in L_e/G_e} tG_ee(KX/KY)$ as $K$-modules, and so (3) of Lemma~\ref{l:induced} is satisfied by $V=KX/KY$.  We deduce that $KX/KY\cong \Ind_e(e(KX/KY))$ by Lemma~\ref{l:induced}, and the result follows.
\end{proof}

Next we observe that the action of $M$ on $M/{\R}$ is cellular in the sense of~\cite{ourmemoirs}.  An order-preserving map $f\colon P\to Q$ of posets is \emph{cellular} if $q\leq f(p)$ implies there is $p'\leq p$ with $f(p')=q$.  A key property of cellular maps is that if $\sigma$ is a chain in $Q$ with largest element $q$ and $f(p)=q$, then there is a chain $\tau$ with largest element $p$ such that $f(\tau)=\sigma$; see~\cite[Lemma~3.4]{ourmemoirs}.  The origin of the term cellular is that if $P,Q$ are CW posets and $f\colon P\to Q$ is a cellular map, then $f$ can be geometrically realized as a regular cellular map between the corresponding regular $CW$ complexes (see~\cite[Lemma~3.5]{ourmemoirs}).

\begin{Prop}\label{p:action.cellular}
The action of a monoid $M$ on $M/{\R}$ is by cellular maps.
\end{Prop}
\begin{proof}
The action of $M$ is clearly order preserving as $aM\subseteq bM$ implies $maM\subseteq mbM$.
Let $m\in M$.  We need to show that if $bM\subseteq maM$, then there is $cM\subseteq aM$ with $mcM=bM$.  But if $bM\subseteq maM$,  then $b=maz$ for some $z\in M$.  It follows that $bM = m(azM)$ and $azM\subseteq aM$.  This completes the proof.
\end{proof}

The next theorem generalizes~\cite[Theorem~5.12]{ourmemoirs}.

\begin{Thm}\label{t:explicit.res.triv}
Let $M$ be a finite regular monoid and $K$ a field. Then the augmented simplicial chain complex $C_\bullet(\Delta(M/{\R}),K)\to K$ is a resolution of the trivial $KM$-module $K$ by modules which are filtered by induced modules.  Moreover, if the characteristic of $K$ divides the order of no maximal subgroup of $M$, and the sandwich matrices of the $\mathscr J$-classes of $M$ are right invertible over the corresponding group algebras, then it is a projective resolution. More generally, under these latter assumptions, if $V$ is a $KG(M)$-module, then $C_\bullet(\Delta(M/{\R}),V)\to V$ is a projective resolution of $V$, where we view $C_q(\Delta(M/{\R}),V)$ as the tensor product $C_q(\Delta(M/{\R}),K)\otimes_K V$ as a $KM$-module.
\end{Thm}
\begin{proof}
The second statement follows from the first. Indeed,  by Theorem~\cite[Theorem~4.4]{rrbg}, if $e\in E(M)$ and $S$ is a simple $KG_e$-module, then $\Ind_e(S)$ is a projective $KM$-module under our hypotheses.  Since $KG_e$ is semisimple, and hence all modules are direct sums of simple modules, it then follows that $\Ind_e(W)$ is projective for any $KG_e$-module $W$.  But a module with a filtration by projective modules is well known to be projective (cf.~\cite[Lemma~5.9]{ourmemoirs}).    The third statement follows from the second since tensoring over $K$ is exact and Theorem~\ref{t:to.triv} implies that $P\otimes_K V$ is projective whenever $P$ is projective.  Thus we prove the first statement.

 Since $M/{\R}$ has maximum element $M$, the category associated to $M/{\R}$ has a terminal object and hence $\Delta(M/{\R})$ is contractible.  Since $M$ acts on $M/{\R}$ via order-preserving maps, we have that $C_\bullet(\Delta(M/{\R}),K)\to K$ is a resolution of the trivial $KM$-module.  It remains to show that each module $C_q(\Delta(M/{\R}),K)$ is  filtered by induced modules.

Fix a principal series for $M$, that is, an unrefinable chain of two-sided ideals
\[\emptyset=I_0\subsetneq I_1\subsetneq\cdots\subsetneq I_n=M.\]
The differences $I_k\setminus I_{k-1}$, with $1\leq k\leq n$, are $\mathscr J$-classes and each $\mathscr J$-class of $M$ appears exactly once in this way; see~\cite[Proposition~1.20]{repbook}.  By regularity, there is an idempotent $e_k\in I_k\setminus I_{k-1}$ for each $1\leq k\leq n$ (cf.~\cite[Proposition~1.22]{repbook}).   Let $P_k = \{mM\mid m\in I_k\}$ be the set of principal right ideals contained in $I_k$.  Then $P_k$ is an $M$-invariant subposet of $P_n=M/{\R}$, which is a lower set in the ordering.  Thus we have a filtration
\[\emptyset = \Delta(P_0)\subsetneq \Delta(P_1)\subsetneq \cdots \subsetneq \Delta(P_n)=\Delta(M/{\R})\] by $M$-invariant induced subcomplexes.  We then have a filtration
\[0=C_q(\Delta(P_0),K)\subseteq C_q(\Delta(P_1),K)\subseteq\cdots\subseteq C_q(\Delta(P_n),K)=C_q(\Delta(M/{\R}),K)\] with subquotients the relative chain spaces $C_q(\Delta(P_k),\Delta(P_{k-1}),K)$, for each $q\geq 0$.

First note that a $q$-simplex $\sigma$ of $\Delta(P_k)$ is of the form
\begin{equation}\label{eq:sigma.form}
m_0M\subsetneq \cdots \subsetneq m_qM
\end{equation}
 with $m_q\in I_k$. Put $m_qM=\max \sigma$.  Let $X$ be the $M$-set of all simplices of $\Delta(P_k)$ of dimension at most $q$ and let $Y$ be the $M$-invariant subset consisting of simplices either belonging to $\Delta(P_{k-1})$ or having dimension less than $q$.  We claim that $C_q(\Delta(P_k),\Delta(P_{k-1}),K)\cong KX/KY$.  Both sets have $K$-bases consisting of the cosets of those $q$-simplices $\sigma$  (oriented in the former case) with $\max \sigma\subseteq I_k$, but $\max\sigma\nsubseteq I_{k-1}$.  Moreover, since $M$ acts by order-preserving maps, if $[\sigma]$ is an oriented $q$-simplex of $\Delta(P_k)$ with $\max \sigma\nsubseteq I_{k-1}$ and $m\in M$, then  \[m([\sigma]+C_q(\Delta(P_{k-1}),K)) = \begin{cases} [m\sigma]+C_q(\Delta(P_{k-1}),K), & \text{if}\ m\sigma\notin Y\\ 0, & \text{else.}\end{cases}\]     Thus $C_q(\Delta(P_k),\Delta(P_{k-1}),K)\cong KX/KY$.  We verify that $(X,Y)$ satisfies the hypothesis of Corollary~\ref{c:action.is.induced} with respect to $e_k$.  It will then follow that $C_q(\Delta(P_k),\Delta(P_{k-1}),K)\cong \Ind_{e_k}(e_kX\setminus e_kY)$.

Suppose that $\sigma\in X\setminus Y$ is as per \eqref{eq:sigma.form} (it must be a $q$-simplex).  Then $m_q\in I_k\setminus I_{k-1}=J_{e_k}$.  Thus we can write $m_q=xe_ky$ with $x,y\in M$. Since $m_qM=x(e_kyM)$ and the action of $x$ on $M/{\R}$ is cellular by Proposition~\ref{p:action.cellular}, we can find a chain $\tau = n_0M\subsetneq \cdots \subsetneq n_qM=e_kyM\subseteq I_k$  with $x\tau = \sigma$ by~\cite[Lemma~3.4]{ourmemoirs}.    Note that $e_k\tau=\tau$ since $n_iM\subseteq e_kyM$ for all $0\leq i\leq q$, whence $\tau\in X$ and $\sigma=xe_k\tau$.  Thus $X\setminus Y\subseteq Me_kX$.

Next we observe that $L(e_k)\subseteq I_{k-1}$.  Indeed, $Me_k\cap J_{e_k}=L_{e_k}$ by the stability of finite monoids~\cite[Theorem~1.13]{repbook}.  Thus $L(e_k)\subseteq I_k\setminus J_{e_k}=I_{k-1}$.   It now follows that $L(e_k)X\subseteq Y$ from the definition.   Suppose that $\sigma,\sigma'\in e_kX\setminus e_kY$ are $q$-simplices, and $m,m'\in L_{e_k}$ with $m\sigma=m'\sigma'$.  Since $e_k\sigma=\sigma$ and $e_k\sigma'=\sigma'$, we have that  $\max \sigma,\max\sigma'\subseteq e_kM$.  But $\max\sigma =m_qM$, $\max\sigma'=m_q'M$ with $m_q,m_q'\in I_k\setminus I_{k-1}=J_{e_k}$.  Since $e_kM\cap J_{e_k}=R_{e_k}$ by stability~\cite[Theorem~1.13]{repbook}, it follows that $\max\sigma=e_kM=\max\sigma'$.   Thus $mM=me_kM=\max m\sigma =\max m'\sigma' = m'e_kM=m'M$.  It follows that $mG_{e_k}=m'G_{e_k}$ by~\cite[Corollary~1.17]{repbook}.  This completes the proof.
\end{proof}

Our first application of Theorem~\ref{t:explicit.res.triv} is to give an explicit projective resolution of length $n$ for the trivial module for $\mathrm{Aff}(n,q)$ in good characteristic.

\begin{Thm}\label{t:explicit.res.triv.aff}
Let $K$ be a field a field whose characteristic does not divide the order of $\mathrm{AGL}(n,q)$.  Let $P(n,q)$ be the poset of nonempty affine subspaces of $\mathbb F_q$, on which $\mathrm{Aff}(n,q)$ acts via direct image.   Then the augmented simplicial chain complex $C_\bullet(\Delta(P(n,q)),K)\to K$ is a projective resolution of the trivial module $K$ of length $n$.
\end{Thm}
\begin{proof}
Proposition~\ref{p:semidirect.directed} and the proof of Theorem~\ref{t:globdim} imply that the hypotheses of  the second statement of Theorem~\ref{t:explicit.res.triv} apply to $\mathrm{Aff}(n,q)$ and $K$.  We observed in the proof of Theorem~\ref{t:globdim}, that $\mathrm{Aff}(n,q)/{\R}\cong P(n,q)$ via the map sending $T\mathrm{Aff}(n,q)$ to the image of $T$, and this map is clearly $M$-equivariant.  The result now follows as $\Delta(P(n,q))$ has dimension $n$.
\end{proof}

Our next goal is to provide a proof of Corollary~\ref{c:main.fd.coind}(4) for finite regular monoids avoiding classifying spaces and spectral sequences.

\begin{Prop}\label{p:homologically.orthogonal}
Let $M$ be a finite regular monoid and $K$ a field whose characteristic does not divide the order of any maximal subgroup of $M$.  Let $V$ be a finite dimensional $KM$-module filtered by induced modules and let $e\in E(M)$.  Then $\Ext^n_{KM}(V,\Coind_e(W))=0$ for any finite dimensional $KG_e$-module $W$ and $n\geq 1$.
\end{Prop}
\begin{proof}
Let $n\geq 1$.
Suppose that $U$ is a $KG_f$-module with $f\in E(M)$. If $MeM\subseteq MfM$, put $I=MeM\setminus J_e$ and note that $I$ is an ideal, and $\Ind_f(U)$ and $\Coind_e(W)$ are $KM/KI$-modules.  Then~\cite[Lemma~16.6]{repbook} yields \[\Ext^n_{KM}(\Ind_f(U),\Coind_e(W))\cong \Ext^n_{KM/kI}(\Ind_f(U),\Coind_e(W)).\]  However, by~\cite[Theorem~16.3]{repbook}, we have $\Ext^n_{KM/KI}(\Ind_f(U),\Coind_e(W))\cong \Ext^n_{KG_e}(e\Ind_f(U),W)=0$, where the last equality holds because $KG_e$ is semisimple.   Suppose that $MeM\nsubseteq MfM$.  Let $I' = MfM\setminus J_f$, which is an ideal, and note that  $\Ind_f(U)$ and $\Coind_e(W)$ are $KM/KI'$-modules because $MeM\nsubseteq MfM$.  Then  we have $\Ext^n_{KM}(\Ind_f(U),\Coind_e(W))\cong \Ext^n_{KM/kI'}(\Ind_f(U),\Coind_e(W))\cong \Ext^n_{KG_f}(U,f\Coind_e(W))=0$ by~\cite[Lemma~16.6]{repbook} and~\cite[Theorem~16.3]{repbook}, as $KG_f$ is semisimple.

We now proceed by induction on the length of a filtration \[0=V_0\subsetneq V_1\subsetneq \cdots\subsetneq V_k=V\] with $V_i/V_{i-1}$ an induced module, for $1\leq i\leq k$.  If $k=0$, there is nothing to prove.    Else, by induction, $\Ext^n_{KM}(V_{k-1},\Coind_e(W))=0$.   Since $V/V_{k-1}$ is induced, we know from the paragraph above $\Ext^n_{KM}(V/V_{k-1}, \Coind_e(W))=0$.   The long exact sequence for $\Ext$ associated to the exact sequence \[0\to V_{k-1}\to V\to V/V_{k-1}\to 0\] yields that $0=\Ext^n_{KM}(V/V_{k-1}, \Coind_e(W))\to \Ext^n_{KM}(V,\Coind_e(W))\to \Ext^n_{KM}(V_{k-1},\Coind_e(W))=0$ is exact for $n\geq 1$.  The result follows.
\end{proof}

The following lemma is elementary.

\begin{Prop}\label{p:tensor.eR/k}
Let $R$ be a ring, $e\in R$ an idempotent and $I$ a right ideal contained in $eR$ with $eReI=I$.  Then for any left $R$-module $M$, there is an isomorphism $(eR/I)\otimes_R M\to eM/IM$ of $eRe$-modules.
\end{Prop}
\begin{proof}
There is a well-defined $eRe$-module homomorphism $(eR/I)\otimes_R M\to eM/IM$ given by $(r+I)\otimes m\mapsto rm+IM$ for $r\in eR$ with inverse given by $m+IM\mapsto (e+I)\otimes m$ for $m\in eM$.
\end{proof}

As a corollary, we deduce the following.

\begin{Cor}\label{c:coind.rep.funct}
Let $M$ be a finite monoid, $e\in E(M)$ and $K$ a commutative ring.  Let $V$ be a $KM$-module and $W$ a $KG_e$-module.  Then there is a natural isomorphism $\Hom_{KM}(V,\Coind_e(W))\cong \Hom_{KG_e}(eV/R(e)V,W)$.
\end{Cor}
\begin{proof}
First note that $eMeR(e)=R(e)$ by the stability of finite monoids~\cite[Theorem~1.13]{repbook}.
Recall that $\Coind_e(W) =\Hom_{KG_e}(eKM/KR(e),W)$.  Thus, by the hom-tensor adjunction and  Proposition~\ref{p:tensor.eR/k}, $\Hom_{KM}(V,\Coind_e(W))\cong \Hom_{KG_e}((eKM/KR(e))\otimes_{KM} V,W)\cong \Hom_{KG_e}(eV/R(e)V,W)$.
\end{proof}

We now give our alternate proof of Corollary~\ref{c:main.fd.coind}(4) for regular monoids.

\begin{Thm}\label{t:main.fd.coind.reg}
Let $M$ be a regular finite monoid with group completion $G(M)$ and $e\in E(M)$.  Let $K$ be a field whose characteristic does not divide the order of any maximal subgroup of $M$,  $V$ a finite dimensional $KG(M)$-module and $W$ a finite dimensional $KG_e$-module.
Then \[\Ext^n_{KM}(V,\Coind_e(W))\cong \Hom_{KG_e}(\til H_{n-1}(\Delta(\Omega(R(e))),K),V^\ast\otimes_K W)\] for all $n\geq 0$.
\end{Thm}
\begin{proof}
Using Corollary~\ref{c:to.homology} and Proposition~\ref{p:hom.commute}, it is enough to prove this when $V$ is the trivial module, that is, to prove $\Ext^n_{KM}(K,\Coind_e(W))\cong  \Hom_{KG_e}(\til H_{n-1}(\Delta(\Omega(R(e))),K),W)$ for all $n\geq 0$.  By Proposition~\ref{p:homologically.orthogonal}, any finite dimensional module that is filtered by induced modules is $F$-acyclic for the functor $F(V) = \Hom_{KM}(-,\Coind_e(W))$.  Hence we can use the resolution from Theorem~\ref{t:explicit.res.triv} to compute $\Ext^n_{KM}(K,\Coind_e(V))$, that is, $\Ext^n_{KM}(K,\Coind_e(V)) \cong H^n(\Hom_{KM}(C_\bullet(\Delta(M/{\R}),K),\Coind_e(W))$.  We have that \[\Hom_{KM}(C_q(\Delta(M/{\R}),K),\Coind_e(W))\cong \Hom_{KG_e}\left(\frac{eC_q(\Delta(M/{\R}),K)}{R(e)C_q(\Delta(M/{\R}),K)},W\right)\] by Corollary~\ref{c:coind.rep.funct}.  We claim that \[eC_q(\Delta(M/{\R}),K)/R(e)C_q(\Delta(M/{\R}),K)\cong C_q(\Delta(\Omega(eM)),\Delta(\Omega(R(e))),K)\] as $KG_e$-modules.  Indeed, trivially $eC_q(\Delta(M/{\R}),K)= C_q(\Delta(\Omega(eM)),K)$ since $e$ retracts $\Delta(M/{\R})$ onto $\Delta(\Omega(eM))$.  It is also clear that if $m\in R(e)$, then $m\sigma$ is a simplex of $\Delta(\Omega(R(e))$ for any simplex $\sigma$.  Conversely, if $\sigma$ is a $q$-simplex of $\Omega(R(e))$ with maximum element $m\in R(e)$, then there is a $q$-simplex $\tau$ with maximum element $1$ with $m\tau=\sigma$ by~\cite[Lemma~3.4]{ourmemoirs} since the action of $m$ is cellular by Proposition~\ref{p:action.cellular}.  Thus $R(e)C_q(\Delta(M/{\R}),K) = C_q(\Delta(\Omega(R(e))),K)$, and so we have isomorphisms commuting with the boundary maps \[eC_q(\Delta(M/{\R}),K)/R(e)C_q(\Delta(M/{\R}),K)\cong C_q(\Delta(\Omega(eM)),\Delta(\Omega(R(e))),K).\]

Since $KG_e$ is semisimple, $W$ is an injective $KG_e$-module.  Therefore,
\begin{align*}
\Ext^n_{KM}(K,\Coind_e(V)) &\cong H^n(\Hom_{KM}(C_\bullet(\Delta(M/{\R}),K),\Coind_e(W))\\
&\cong H^n(\Hom_{KG_e}(C_\bullet(\Delta(\Omega(eM)),\Delta(\Omega(R(e))),K),W)\\
&\cong \Hom_{KG_e}(H_n(\Delta(\Omega(eM)),\Delta(\Omega(R(e))),K),W)\\ &\cong  \Hom_{KG_e}(\til H_{n-1}(\Delta(\Omega(R(e))),K),W)
\end{align*}
with the last isomorphism coming via the connecting homomorphism from the long exact sequence in relative homology (which is natural, and hence a $KG_e$-module homomorphism), as $\Delta(\Omega(eM))$ is contractible because $\Omega(eM)$ has maximum element $eM$.  This completes the proof.
\end{proof}
\def\malce{\mathbin{\hbox{$\bigcirc$\rlap{\kern-7.75pt\raise0,50pt\hbox{${\tt
  m}$}}}}}\def\cprime{$'$} \def\cprime{$'$} \def\cprime{$'$} \def\cprime{$'$}
  \def\cprime{$'$} \def\cprime{$'$} \def\cprime{$'$} \def\cprime{$'$}
  \def\cprime{$'$} \def\cprime{$'$}

%\bibliographystyle{abbrv}
%\bibliography{standard2}

\end{document}